\documentclass{amsart}

\usepackage{epsfig}
\usepackage{mathtools}
\usepackage{amsmath}
\usepackage{amssymb}
\usepackage{amscd}
\usepackage{graphicx}
\usepackage{color}
\usepackage{graphs}

\newtheorem{thm}{Theorem}[section]
\newtheorem{lem}[thm]{Lemma}

\newtheorem{fact}[thm]{Fact}

\newtheorem{ques}[thm]{Question}

\newtheorem{defn}[thm]{Definition}

\theoremstyle{remark}
\newtheorem{rem}[thm]{Remark}

\newtheorem{exam}[thm]{Example}

\newcommand{\im}{invariant measure}
\newcommand{\sq}{sequence}

\newcommand{\z}{\mathbb Z}
\newcommand{\na}{\mathbb N}

\newcommand{\s}{\mathbb S}
\newcommand{\p}{\mathcal P}
\newcommand{\U}{\mathcal U}

\newcommand{\R}{\mathcal R}
\newcommand{\Pa}{\mathcal P}
\newcommand{\B}{\mathcal B}
\newcommand{\htop}{\mathbf h_{\mathsf{top}}}

\newcommand{\tl}{topological}
\newcommand{\ds}{dynamical system}

\newcommand{\xt}{$(X,T)$}
\newcommand{\xtp}{$(X',T')$}
\newcommand{\ys}{$(Y,S)$}
\newcommand{\xmt}{$(X,\mu,T)$}

\newcommand{\yns}{$(Y,\nu,S)$}
\newcommand{\mtx}{\mathcal M_T(X)}
\newcommand{\msy}{\mathcal M_S(Y)}
\newcommand{\diam}{\mathsf{diam}}

\newcommand{\Per}{\mathsf{Per}}
\newcommand{\id}{\mathsf{id}}

\numberwithin{equation}{section}

\begin{document}

\baselineskip=12pt
\title{Dynamics in dimension zero\\ \small A survey}

\author{Tomasz Downarowicz}
\address{Faculty of Mathematics and	Faculty of Fundamental Problems of Technology, Wroclaw University of Technology, Wroclaw, Poland}
\email{Tomasz.Downarowicz@pwr.edu.pl}

\author{Olena Karpel}
\address{B. Verkin Institute for Low Temperature Physics and Engineering, Kharkiv, Ukraine\\
{\em Current address: Department of Dynamical Systems, Institute of Mathematics of Polish Academy of Sciences, Wroclaw, Poland}}
\email{helen.karpel@gmail.com}

\begin{abstract}
The goal of this paper is to put together several techniques in handling dynamical systems on zero-dimensional spaces, such as array representation, inverse limit representation, or Bratteli--Vershik representation. We describe how one can switch from one representation to another. We also briefly review some more recent related notions: symbolic extensions, symbolic extensions with an embedding, and uniform generators. We devote a great deal of attention to marker techniques and we use them to prove two types of results: one concerning entropy and vertical data compression, and another, about the existence of isomorphic minimal models for aperiodic systems. We also introduce so-called decisiveness of Bratteli--Vershik systems and give for it a sufficient condition.
\end{abstract}

\maketitle
\markboth{}{}

\section{Introduction}
Zero-dimensional dynamical systems is a huge class, which is in many aspects universal, i.e., every \ds\ has a zero-dimensional ``counterfeit'' possessing most of its interesting dynamical properties. The fundamental result of this kind is the Jewett-Krieger Theorem \cite{J70, Kr72}, which associates to every ergodic automorphism a strictly ergodic (minimal and uniquely ergodic) topological model, which is in fact zero-dimensional. The same fact has been proved for endomorphisms by A. Rosenthal \cite{Ro88}. In the category of smooth \ds s, zero-dimensional systems often appear as attractors or supports of \im s, i.e., as subsystems, where most of the important dynamics happens. The Julia set, in some classes of meromorphic
functions on Riemann surfaces, is zero-dimensional (see e.g. \cite{DHSi14}).
From the point of view of entropy it is fully sufficient to understand zero-dimensional dynamics. This is due to the fact, that every \tl\ \ds\ admits a zero-dimensional faithful (every \im\ lifts to a unique preimage) and principal (preserving entropy of each \im) extension (see \cite{DH12}). Such an extension inherits all imaginable entropy properties of the underlying system.
\smallskip

But there is still a lot to do, as far as finding zero-dimensional representations is concerned. It is a pending open question whether all of topological dynamics can be realized in dimension zero, up to sets of measure zero for all \im s (under the obvious assumption that periodic points form a zero-dimensional subset). One of possible solutions could be finding, for every such system, an isomorphic zero-dimensional extension, i.e., one which is faithful and such that each ergodic measure is dynamically isomorphic to its unique preimage. This problem remains wide open since many decades (the positive answer is known only for systems which have so-called small boundary property, see \cite{L89,LW00}).
\smallskip

The class of zero-dimensional systems includes systems studied since the origins of ergodic theory and discrete dynamical systems, namely subshifts. They are the ``building blocks'' of all zero-dimensional dynamical systems, and symbolic dynamics is the branch of dynamics involved with subshifts. Historically, a subshift appeared (implicitly) for the first time in a work of Jacques S. Hadamard about geodesic flows on surfaces with negative curvature \cite{Had98}. The idea was later developed by Harold M. Morse, leading to the creation of symbolic dynamics, which becomes a separate field after the publication of his joint book with Gustav A. Hedlund \cite{MH38}. For a more contemporary approach to symbolic dynamics see also \cite{LM95}.
Subsequently, subshifts have been studied for their algebraic, combinatorial, number-theoretic, measure-theoretic, entropy, and complexity aspects, and many more, whose list is practically endless. It is impossible not to mention here Fustenberg's correspondence principle \cite{F77, F81}, which allows to study many combinatorial properties of subsets of natural numbers by viewing the subshifts generated by their characteristic functions treated as elements of the $\{0,1\}$-symbolic space. Ramsey Theory is the field where this principle finds most of its applications (see e.g. \cite{R30,Ber96,Ber03}). Subshifts play the key role in the analysis of the majority of other dynamical systems. As an off hand evidence let us mention here the analysis of the logistic family with help of the kneading sequence and its generated subshift (see e.g. \cite{MT88,CE80}). Subshifts are the most manageable in constructing systems exhibiting rare phenomena, hence they appear naturally in many examples. The same applies to the more general class of zero-dimensional systems. For instance, it has been proved that any Choquet simplex can be realized as the set of \im s in some minimal subshift (\cite{Do91}).
This result has been later extended to realizing also any prescribed affine nonnegative function as the entropy function on \im s (upper semicontinuous in subshifts, and of the Young class LU in zero-dimensional systems, see \cite{DS03}).

For the above reasons, it is desirable to find, for a general \tl\ \ds, not only its zero-dimensional, but preferably a symbolic extension, which would reflect its dynamical properties in the best possible way. The existence of just any symbolic extension has led to the development of the theory of symbolic extension entropy \cite{BD05, Do05}. The existence of a principal symbolic extension turns out to be equivalent to an older notion of asymptotic $h$-expansiveness introduced in \cite{Mi76}. Since J. Buzzi \cite{Buz97} proved that every $C^\infty$ diffeomorphism of a smooth Riemannian manifold is asymptotically $h$-expansive, the notion gained a lot of attention. Further, with an (inevitable) additional assumption on periodic points, asymptotic $h$-expansiveness yields the existence of not only a principal but in fact an isomorphic symbolic extension \cite{B15}, which is the strongest possible connection to a symbolic system (not counting \tl\ conjugacy). For systems which are not asymptotically $h$-expansive, it is still possible to have a symbolic extension \emph{with an embedding}, i.e. one which contains a noncompact (measurable) subsystem which is an isomorphic extension of the given system. The existence of such an extension can be expressed in terms of uniform generators, i.e.,  measurable partitions which separate points in a uniform manner (see \cite{BD16}).

Among zero-dimensional systems which are not subshifts, perhaps the most studied class is that described by Bratteli diagrams. These were introduced in \cite{Br72} for the classification of approximately finite (AF) $C^*$-algebras, and showed to be extremely useful in Cantor, Borel and measurable dynamics. The ideas of A. Vershik \cite{vershik:1981, vershik:1982} led to a realization of any ergodic automorphism of the standard measure space as a transformation acting on the (always zero-dimensional) path space of a Bratteli diagram. Dynamical systems obtained in this way are called Bratteli--Vershik models. Later, R. Herman, I. Putnam, and C. Skau \cite{HPS92} proved that any minimal homeomorphism of a Cantor set can be realized as a Vershik map acting on the path space of a Bratteli diagram. Such realization provided a convenient tool for describing the simplex of invariant probability measures (see e.g. \cite{BKMS_2010,BKMS_2013,ABKK16}) and orbit equivalence classes (see \cite{giordano_putnam_skau:1995,glasner_weiss:1995-1,giordano_putnam_skau:2004,GMPS08,GMPS10,HKY12}). As for applications of Bratteli diagrams in Borel dynamics see, for instance, \cite{BDK06}.
\medskip

The goal of this paper is to put together several techniques in handling zero-dimensional dynamical systems, such as array representation, inverse limit representation, or Bratteli--Vershik representation. We describe how one can switch from one representation to another. We also briefly review symbolic extensions, symbolic extensions with an embedding, and uniform generators. We devote a great deal of attention to marker techniques, based on a lemma attributed to Krieger (which can be found in \cite{Bo83}). Markers have shown extremely useful in data compression and entropy calculations. They are heavily used in creating symbolic extensions. In this paper we use them to prove two types of results: one concerning entropy and vertical data compression, and another, about the existence of isomorphic minimal models for aperiodic systems. Finally, we discuss Bratteli--Vershik models and for these we introduce the notion of decisiveness. We give a sufficient condition for a zero-dimensional system to posses a decisive Bratteli--Vershik model.

\section{Basic concepts}

This section contains the description and basic connections between notions crucial in zero-dimensional dynamics.

\subsection{Subshifts, symbolic extensions, uniform generators}

Subshifts are the most elementary zero-dimensional systems. In this subsection we explain when, how and to what extent an abstract system can be modeled by a subshift. In the sequel we will consider \tl\ \ds s \xt, where $X$ is a compact metric space and $T:X\to X$ is a continuous map. The map $T$ generates the action of the semigroup $\na_0$ of non-negative integers on $X$, by the formula $n\mapsto T^n$ ($n\in\na_0$). If $T$ is a homeomorphism, it also generates (by the same formula) the action of the group $\z$ of all integers. Since in most of our considerations we do not want to restrict to just one of the two possible actions, we will use $\s$ to denote any element of the set $\{\na_0,\z\}$.\footnote{Many of the techniques presented in this paper apply equally well to actions of other groups or semigroups, however, in this note we focus on actions of a single transfromation.}

\begin{defn}
Let $\Lambda$ be finite set (called an \emph{alphabet}) endowed with the discrete topology. By a \emph{subshift} over $\Lambda$ we mean any closed, shift-invariant subset $X\subset\Lambda^\s$, endowed with the product topology (and some compatible metric). Each element of $X$ is pictured as a \sq\ $x=(x_n)_{n\in\s}$ of \emph{symbols} from $\Lambda$. The subshift is by default regarded with the action
of the shift transformation $S$ defined by $Sx=y$, where $\forall n\in\s\ y_n = x_{n+1}$.
\end{defn}

\begin{defn} Let $\p$ be a finite Borel-measurable partition of $X$ and let $\Lambda$ be a finite alphabet which bijectively labels the elements of $\p$, i.e., $\p = \{P_a:a\in\Lambda\}$. By a $\p$-name of $x\in X$ we shall mean the \sq\ $(x_n)_{n\in\s}\in \Lambda^\s$ obtained by the rule
$$
x_n = a \iff T^n(x)\in P_a.
$$
\end{defn}
Notice that the closure of all $\p$-names is a subshift over $\Lambda$.

\begin{defn}
A \tl\ \ds\ \xt\ (with the action of $\s$) is \emph{expansive} (with expansive constant $\epsilon > 0$) if
for any $x,y\in X$
$$
x\neq y\implies \exists n\in\s\ \ d(T^nx,T^ny)\ge\epsilon.
$$
\end{defn}

\begin{thm}[Hedlund's theorem \cite{H69}]
A system is (\tl ly conjugate to) a subshift if and only if it is zero-dimensional and expansive.
\end{thm}

\begin{proof}
The implication $\implies$ is obvious. Conversely, let $\epsilon$ be the expansive constant. There exists a finite clopen partition $\p$ of $X$ into sets of diameter $<\epsilon$. Now, the map $x\mapsto\p$-name of $x$ is continuous and injective and intertwines the action with the shift. So, it is a conjugacy with the
image, which is a subshift over an alphabet $\Lambda$ bijectively labeling $\p$.
\end{proof}

It follows from the above, that the range of applicability of topological conjugacy as a tool for symbolic representation of \ds s is very limited. A system conjugate to a subshift is practically a subshift. Much wider variety of \ds s can be symbolically represented by means of symbolic extensions. A symbolic extension of a system \xt\ carries all the information, both measure-theoretic and \tl, necessary to reconstruct \xt. Unfortunately, typically it contains also a large amount of superfluous information, seemingly useless for this task. A natural desire is to maximally reduce this amount. One of ways is by minimizing the \tl\ entropy of the extension. Faithfulness of the extension is also a desirable property, as is reduces the amount of preimages of each orbit. Let us give the formal definitions.

\begin{defn}
Let \xt\ be a \tl\ \ds. By a \emph{symbolic extension} we mean a subshift \ys\ together with a \tl\ factor map $\pi:(Y,S)\to(X,T)$.
\end{defn}

\begin{defn}
An extension $\pi:(Y,S)\to(X,T)$ is \emph{faithful} if the adjacent map on \im s $\nu\mapsto\pi(\nu)$
defined by $\pi(\nu)(A)=\nu(\pi^{-1}(A))$ (which is always a surjection from the set $\msy$ of $S$-invariant probability measures on $Y$ to the analogous set $\mtx$) is injective.
\end{defn}

The existence and smallest possible entropy of a symbolic extension of a given system \xt\ is in general difficult to decide. It is subject of the \emph{theory of symbolic extensions and entropy structures} (see e.g. \cite{BD05} or \cite{Do11}). It depends on subtle entropic properties of \im s. Typically, unless the underlying system
\xt\ has a special property called asymptotic $h$-expansiveness, any symbolic extension \ys\ will either have larger \tl\ entropy than \xt, or at least some \im s on $Y$ will have larger entropy than their images operating on $X$. We choose to skip the technical definition of asymptotic $h$-expansiveness and refer to the original paper \cite{Mi76}. Let us only formulate the associated result, which was first proved in \cite{BFF02}:

\begin{thm}\label{ahe}
A system \xt\ is asymptotically $h$-expansive if and only if it admits a symbolic extension $\pi:(Y,S)\to(X,T)$ such that for every $\nu\in\msy$, $h_\nu(S) = h_\mu(T)$, where $\mu=\pi(\nu)$.
\end{thm}

An extension as above is called \emph{principal}. In \cite{Ser12}, J. Serafin was able to show that every symbolic extension\footnote{In Serafin's paper, there is aperiodicity assumption. But later, in \cite{DH12}, it was shown that this assumption can be dropped.} can be replaced by one which is faithful and has the same entropy properties. So, the principal extension in Theorem \ref{ahe} can be chosen faithful. A principal and faithful symbolic extension can be considered a symbolic model which is perfect from the point of view of information theory. However, it is not perfect from the point of view of dynamics, because preservation of entropy (for each \im) is far from preservation of dynamics. The most desirable situation occurs when the symbolic extension is \emph{isomorphic}, as defined below:

\begin{defn}
An extension $\pi:(Y,S)\to(X,T)$ (not necessarily symbolic) is called \emph{isomorphic} is it is faithful
and for every $\nu\in\msy$ the map $\pi$ serves as an isomorphism between the measure-preserving systems $(Y,\nu,S)$ and $(X,\mu,T)$ (here $\mu=\pi(\nu)$ and we have skipped the Borel sigma-algebras in the denotation).
\end{defn}

An isomorphic extension is the best symbolic representation of a non-symbolic (not expansive or not zero-dimensional) system one can imagine. It preserves the affine-topological structure of the simplex of \im s and the dynamics (up to measure-theoretic isomorphism) individually for each \im. The existence of such a perfect symbolic extension is described by the following strengthening, proved by D. Burguet, of Theorem \ref{ahe} (see \cite{B15}):

\begin{thm}
If \xt\ is asymptotically $h$-expansive and aperiodic (i.e., contains no periodic orbits) then it admits an isomorphic symbolic extension.
\end{thm}

Note that since any isomorphic extension is principal, the converse (for aperiodic systems) is already implied by Theorem \ref{ahe}. We remark, that in Burguet's paper, aperiodicity is relaxed to a technical condition on the distribution of periodic points, which we have decided to skip in this survey.

Once again, we can see that as perfect symbolic extensions as isomorphic or even just principal, are available only for the relatively narrow class of systems which are asymptotically $h$-expansive (in smooth dynamics, this roughly corresponds to the class $C^\infty$, see \cite{Buz97}). This is why we introduce a weaker relation between a system and its extension, yet allowing to consider the symbolic extension of this kind a very good symbolic representation. The underlying system is not only a \tl\ factor but in some sense also a Borel-measurable subsystem. Another way of saying this is that the symbolic extension contains a \emph{non-compact} isomorphic extension of \xt. The ``perfectness'' of such an extension however is ``spoiled'' by the remaining part, which is merely as good as a usual symbolic extension. The formal definition reads:

\begin{defn}
Let \xt\ be a \tl\ \ds. By a \emph{symbolic extension with an embedding} we mean a symbolic extension $\pi:(Y,S)\to(X,T)$, which admits an \emph{equivariant Borel selector from preimages}, i.e., a measurable map $\psi:X\to Y$ such that $\psi\circ T = S\circ\psi$ and $\pi\circ\psi = \id_X$. In fact, $\psi$ is a Borel embedding of the \ds\ \xt\ into \ys, hence the name.
\end{defn}

In search for the range of applicability of this kind of symbolic extensions, we first
observe the following generalization of Hedlund's theorem:

\begin{thm}
Every expansive system has a symbolic extension with an embedding.
\end{thm}

\begin{proof}
Let $\epsilon$ be the expansive constant. There exists a finite partition $\p$ of $X$ into Borel-measurable sets of diameter $<\epsilon$. Now, the map $\psi$ defined by $x\mapsto\p$-name of $x$ is an injective and action-preserving measurable map into the full shift over an alphabet $\Lambda$ bijectively labeling $\p$. Let $Y=\overline{\psi(X)}$. Clearly, $Y$ is a subshift. For $y\in Y$ and $n\in\s$ let $C_n(y)$ denote the set of points whose $\p$-name coincides on the interval $\{k:|k|\le n\}$ with the block of $y$ over the same interval. Since the same block appears in $\psi(x)$ for some $x\in X$, it is clear that $C_n(y)$ is nonempty. By compactness, the intersection $C(y) = \bigcap_n \overline{C_n(y)}$ is nonempty. Note also that
$C_n(y)\subset \bigcap_n T^{-n}(\overline{P_{y_n}})$. Since for any distinct points $x,x'$ there is $n$ such that $T^nx, T^ny$ are $\epsilon$-apart, these two points must not belong to the closure of the same element of $\p$.  This easily implies that, for each $y\in Y$, the set $C(y)$ consists of one point. Let us denote this point by $\pi(y)$. It is now an elementary fact in topology that the diameters of the sets $\overline {C_n(y)}$ must shrink to zero, which implies that the mapping $\pi:Y\to X$ is continuous. Also, it preserves the action, hence it is a \tl\ factor map. Clearly, $\psi$ is a measurable selector from its preimages.
\end{proof}

Clearly, every isomorphic extension is one with an embedding, so that expansiveness is by far too strong
in the above theorem (it is stronger than asymptotic $h$-expansiveness). In fact, we have included the above proof only in order to exercise the construction of a symbolic extension with an embedding in an easy case. The following theorem, although has a similar proof, is a much more precise and gives an equivalent condition. Before we formulate it, we need a definition:

\begin{defn}Let \xt\ be a \tl\ \ds.
A finite measurable partition $\p$ of $X$ satisfying
$$
\lim_n\diam(\p^n) = 0
$$
will be called a \emph{uniform generator}.

Notation: $\p^n$ denotes the common refinement $\bigvee_{|i|< n}T^{-i}(\p)$,

$\diam(\p)$ denotes the maximal diameter of an atom of $\p$.
\end{defn}

Possessing such a generator is in a sense analogous to being expansive. In fact, one might have an impression that it is the same. The irrational rotation is the easiest counterexample (the partition indeed separates orbits, but there is no associated expansive distance). The connection of this notion with our topic is very strong:

\begin{thm}[\cite{BD16}]
A system \xt\ has a uniform generator if and only if it has a symbolic extension with an embedding.
\end{thm}

\begin{proof}
The proof of the first implication is almost the same as in the preceding theorem. The only difference is that since the sets $C_n(y)$ are elements of the partition $\p^n$, the shrinking to zero of their diameters (which are the same as of their closures) follows directly from the assumption.

Now suppose that \xt\ has a symbolic extension $\pi:(Y,S)\to(X,T)$ admitting a required selector $\psi$. Let $\p_\Lambda$ denote the ``zero-coordinate partition'' of $Y$ (i.e., the partition into cylinder sets corresponding to single symbols from $\Lambda$), and define $\p=\psi^{-1}(\p_\Lambda)$. Clearly, $\p$ is a measurable partition of $X$. The convergence of the diameters of $\p^n$ to zero follows directly from the three facts: that the same property has $\p_\Lambda$ in $Y$, that each atom of $\p^n$ is contained in the image by $\pi$ of an atom of $\p_\Lambda^n$, and that $\pi$ is uniformly continuous.
\end{proof}

In the above mentioned paper \cite{BD16} there is given a characterization of systems which admit symbolic extensions with an embedding, as well as tools are provided for computing the lowest possible entropy of such an extension. This is done in terms similar to those used in the entropy theory of general symbolic extensions. In fact, if the system is aperiodic, the embedding requirement has no influence on the entropy of the extension or on the necessary increase of entropy for each individual \im. It does affect, however, faithfulness for quite obvious reasons: a typical measure must have at least two preimages in the symbolic extension: one to which it is isomorphic, and another which has the inevitably increased entropy.

When the system does have periodic points, the criteria become much more complicated and usually a symbolic extension with an embedding will have larger entropy than one without. In this case matters become way too complicated to even be sketched here.

This is as far as we have decided to go with reviewing symbolic extensions in this survey. For more, the reader is referred to the above cited papers.

\subsection{Universality of zero-dimensional dynamics}
Unlike in the case of symbolic extensions, the applicability of zero-dimensional extensions which preserve the dynamics has a very wide range. In fact, every \tl\ \ds\ \xt\ admits a principal and faithful zero-dimensional extension, as stated in the theorem below:

\begin{thm}\cite{DH12}
Any \tl\ \ds\ \xt\ has a faithful principal zero-dimensional extension. Moreover, there exists such an
extension with no periodic points.
\end{thm}

As far as isomorphic zero-dimensional extensions are concerned, there is one obvious constraint on the system \xt: the set of periodic points must be zero-dimensional. But even for aperiodic systems the problem of the existence of such an extension is currently open. The answer is known (and positive) in a large class of systems which satisfy so-called small boundary property, originating from \cite{L89}:

\begin{defn}
A system \xt\ has \emph{small boundary property} if there exists a base of the topology consisting of sets whose boundaries are null sets, i.e., have measure zero for every \im. Equivalently, there exists a refining \sq\ of partitions\footnote{A \sq\ of partition $\{\Pa_k\}_{k\ge 1}$ is \emph{refining} if $\Pa_{k+1}\preccurlyeq\Pa_k$ for each $k$, and $\diam(\Pa_k)\to 0$. The easy proof of the equivalence in the definition is left to the reader.} into sets with null boundaries.
\end{defn}

It follows from the works of E. Lindenstrauss and B. Weiss \cite{L89,LW00}, that every system with finite topological entropy and possessing an infinite minimal factor, has small boundary property. It is not known to what extent can the latter assumption be weakened. Clearly, some assumption is necessary, as for instance the interval with the action of the identity map does not have small boundary property. For instance, Kulesza \cite{Ku95} proved this property for any finite-dimensional system with zero-dimensional set of periodic points (regardless of entropy). In any case, the class of systems with small boundary property is quite large, which makes the following easy observation very useful:

\begin{thm}
If \xt\ has the small boundary property then it admits an isomorphic zero-dimensional extension.
\end{thm}

\begin{proof}[Proof adapted from \cite{BD05}]
Let $\{\Pa_k\}_{k\ge 1}$ be a refining sequence of finite partitions with null boundaries.
For each $k\ge 1$ let $(x_{k,n})_{n\in\s}$ be the $\Pa_k$-name of $x$. In this manner, we associate to each $x$ an array $\phi(x)=(x_{k,n})_{k\ge 1, n\in\s}$. Let $X'$ be the closure of $\phi(X)$. By a standard argument, there exists a continuous factor map $\pi:X'\to X$, which on $\phi(X)$ is inverse to $\phi$. In other words, $\pi$ is an extension with an embedding. However, a point $x\in X$ has multiple preimages in $X'$ only if its orbit visits a boundary of an element of some partition $\Pa_k$. Clearly, the set of such points is null, so the extension is in fact isomorphic.
\end{proof}

As we said, it is unknown whether all aperiodic systems have isomorphic zero-dimensional extensions, and whether aperiodicity can be relaxed. On the other hand, it is not known whether the existence of such an extension is equivalent to the small boundary property.

But even a weaker relation between a system \xt\ and a zero-dimensional ``partner'', say \xtp, would be satisfactory, and we mean here the existence of a common isomorphic extension (not necessarily zero-dimensional). The two systems would then be joined by a measurable map $\pi$ (defined except on a null set) which applied to \im s would serve as an affine homeomorphism, while for each \im\ and its image, it would serve as a measure-theoretic isomorphism. It is unknown whether every aperiodic system has a zero-dimensional partner of this kind but we are inclined to believe that it is so. Let us summarize the above mentioned problems:

\begin{ques}
\begin{enumerate}
	\item Does every aperiodic system have an isomorphic zero-dimensional extension? 
	\item If yes, how can the condition of aperiodicity be relaxed? 
	\item If not, is the existence of an isomorphic extension equivalent to the small boundary property? 
	\item Is it at least true that for any aperiodic system \xt\ there exists a zero-dimensional system \xtp\ such that there is a common isomorphic extension of both \xt\ and \xtp?
	\item If yes, how can the condition of aperiodicity be relaxed? 
\end{enumerate}
\end{ques}

\subsection{Array systems, countable joinings and inverse limits of subshifts}

Now we provide a convenient representation of all zero-dimensional systems and we explain how they can be built from subshifts. The notions introduced in this subsection are crucial for managing zero-dimensional dynamics in its full generality.

\begin{defn}
Let $\Lambda_1,\Lambda_2,\dots$ be finite alphabets (the cardinalities need not be bounded).
By an \emph{array system} we mean any closed, shift-invariant subset of the Cartesian product
$\prod_k \Lambda_k^\s$. Each element of the array system can be pictured as an array
$x=[x_{k,n}]_{k\in\na,n\in\s}$, such that each $x_{k,n}$ belongs to $\Lambda_k$.
We denote by $\phi_k$ and $\pi_k$ the projections of $X$ to $\Lambda_k^\s$ and $\prod_{i\le k}\Lambda_k^\s$, respectively. The images by $\phi_k$ and $\pi_k$ will be denoted as $X_k$
and $X_{[1,k]}$, and called the \emph{$k$th row factor} and the \emph{top $k$ rows factor} of $X$,
respectively. The $k$th row factor is a subshift over $\Lambda_k$, while the top $k$ rows factor
is a subshift over the finite alphabet $\Delta_k=\Lambda_1\times\cdots\times\Lambda_k$. Speaking about an array we will refer to the indices $k$ and $n$ as \emph{vertical} and \emph{horizontal} coordinates (positions), respectively.
\end{defn}

\begin{defn}
Let $(X_k,T_k)$ be a (finite or countable) \sq\ of \tl\ \ds s. By a \tl\ joining of these systems
we mean any closed subset of the product $\prod_k X_k$ which has full projections on every coordinate
and is invariant under the product transformation $T=T_1\times T_2\times\cdots$.
\end{defn}

A special case of a countable joining is an inverse limit.
\begin{defn}
Let $(X_k,T_k)$ be countable \sq\ of \tl\ \ds s such that $(X_k,T_k)$ is a topological factor of $(X_{k+1},T_{k+1})$, for each $k\ge 1$. The corresponding (surjective) factor maps $\psi_k:X_{k+1}\to X_k$ are called the \emph{bonding maps}. The \emph{inverse limit} of these systems is their joining defined by the rule
$$
(x_k)_{k\ge1}\in\prod_kX_k \text{ \ belongs to the inverse limit $X$ if and only if \ } \forall k \; x_k = \psi_k(x_{k+1}).
$$
The inverse limit is denoted by $(X,T)=\overset{\longleftarrow}\lim_k(X_k,T_k)$.
\end{defn}

\begin{thm}
The following statements about a \tl\ \ds\ \xt\ are equivalent:
\begin{enumerate}
\item $X$ is zero-dimensional,
\item \xt\ is (conjugate to) an array system,
\item \xt\ is (conjugate to) a countable joining of subshifts,
\item \xt\ is (conjugate to) an inverse limit of subshifts.
\end{enumerate}
\end{thm}

The elementary proof is left to the reader. From this place on we will present almost all proofs, some of them will be quoted from original papers (mostly in modified versions), some will be completely new.

\subsection{Markers in aperiodic systems}
In this section we provide a very useful tool in manipulating zero-dimensional systems.

\begin{defn}
Let \xt\ be a \tl\ \ds\ and let $n\in\na$. By an \emph{$n$-marker} we mean a clopen set $F\subset X$
such that
\begin{enumerate}
	\item no orbit visits $F$ twice in $n$ steps  (i.e., $F,T^{-1}F,\dots,T^{-(n-1)}F$ are disjoint;
	we will say that $F$ is \emph{$n$-separated}),
	\item every orbit visits $F$ at least once (by compactness, this implies that for some $N\in\na$,
	we have $F\cup T^{-1}F\cup \dots \cup T^{-(N-1)}F=X$).
\end{enumerate}
\end{defn}

We note that a notion of a marker which is an open set, and not necessarily clopen,
 is investigated in \cite{Gu15,Gu17}. Such a concept may be useful for the study of non zero-dimensional minimal dynamical systems.

We have the following key fact:
\begin{thm}[Krieger's Marker Lemma, aperiodic case, see \cite{Bo83}]\label{krieger}
If \xt\ is an aperiodic (with no periodic points) and zero-dimensional system then for every $n$ there exists an $n$-marker. The parameter $N$ can be selected equal to $2n-1$.
\end{thm}

\begin{proof}
Fix an $n\in\na$. Because there are no periodic points, every point $x\in X$ has a clopen neighborhood $U_x$ which is $n$-separated. Choose a finite subcover $\U = \{U_j:j=1,\dots,m\}$ of the cover by the sets $U_x$. The sets $U_j'= T^{-nm}(U_j)$ (where $m$ is the cardinality of $\U$) are also $n$-separated, they cover $X$, and, in addition their $nm$ forward images are also clopen. Now, we define inductively
\begin{align*}
&F_1 := U'_1\\
&F_{j+1}:= F_j\cup\Bigl(U'_{j+1} \setminus \bigcup_{-n< i< n}T^{-i}(F_j)\Bigr)
\end{align*}
and we set $F = F_m$. This is clearly a clopen set. Also note that the \sq\ $F_j$ increases with $j$.
\smallskip

For (1) we will inductively prove that each $F_j$ is $n$-separated. For $F_1$ this is clear because $F_1=U'_1$ is $n$-separated. Suppose we have proved the property for $F_{j-1}$. If $n$-separation fails
for $F_j$ then $T^{-i'}F_j$ and $T^{-i''}F_j$ are not disjoint for some $0\le i'<i''< n$. Applying
$T^{i'}$ we get that $F_j$ and $T^{-i}F_j$ contain a common point $x$, where $i=i''-i'$ satisfies
$0\le i< n$. In other words, $F_j$ contains both $x$ and $T^ix$. Since $F_j$ is contained in the union of two $n$-separated sets $F_{j-1}$ and $U'_j$, none of these sets contains both $x$ and $T^ix$. So, one of these points belongs to $F_{j-1}$ and the other to $U'_j$ (there are two possible cases). But, in either case, the point which belongs to $U'_j$ also belongs to the set $\bigcup_{-n< i< n}T^{-i}(F_{j-1})$, subtracted from $U'_j$ when defining $F_j$. So, that point does not belong to $F_j$, a contradiction.
\smallskip

For~(2) consider a point $x\in X$. Then $T^{n-1}x$ belongs to some $U'_j$ $(1\le j \le m)$. If $T^{n-1}x\in F_j$ then $x\in T^{-(n-1)}F_j\subset T^{-(n-1)}F$ and (2) holds. The only way $T^{n-1}x$ may not belong to $F_j$ is that $j>1$ and $T^{n-1}x$ belongs to $\bigcup_{-n< i< n}T^{-i}(F_{j-1})$. In such case, however, $x\in T^{-(n-1+i)}(F_{j-1})\subset T^{-(n-1+i)}F$ where $n-1+i\in\{0,\dots,2n-2\}$, as required.
\end{proof}

In zero-dimensional systems represented as array systems, the times of visits in the marker sets can be conveniently pictured as additional symbols (which we will call \emph{markers}) inserted in the rows of each array. We choose to use vertical bars separating the symbols in a selected row. If $T^nx$ belongs the marker set, we will put such a marker between the symbols at positions $n$ and $n+1$.
Because the marker sets are clopen, adding the markers produces a topologically conjugate representation of the system. In this form the markers can be easily manipulated (shifted, added, removed, copied from one row to another). Every such manipulation translates to (complicated) set operations on the marker sets, but the array representation enables one to forget these complications. In order to keep our system conjugate, we only need to make sure that our manipulations are
\begin{itemize}
	\item shift equivariant, and
	\item depend locally on a bounded area in the array only (this is continuity).
\end{itemize}

\begin{thm}\label{markers}
Every aperiodic, zero-dimensional system \xt\ admits a \emph{standard markered array representation} such that in every array $x\in X$ the following restrictions hold
\begin{enumerate}
	\item the markers in row $k+1$ are allowed only at horizontal positions of the markers in row $k$ (i.e., 
	the corresponding marker sets are nested),
	\item the markers in row $k$ appear with gaps ranging between two positive integers $n_k^{\min}\le
	n_k^{\max}$, where $\lim_k n_k^{\min}=\infty$.
\end{enumerate}
Additionally, we can arrange the system of markers to be \emph{balanced}, i.e., so that
\begin{enumerate}
	\item[(3)] the ratios $\frac{n_k^{\min}}{n_k^{\max}}$ tend to $1$ with $k$.
\end{enumerate}
\end{thm}

The markers in row $k$ will be called $k$-markers. Formally there is a notational collision with ``$n$-marker sets'' (for instance the $k$-markers correspond to $n_k$-marker sets), but we will never exchange the letters $n$ and $k$ in their roles. The finite subarrays stretching vertically through rows $1$ through $k$ and horizontally between two consecutive $k$-markers of some array $x\in X$ will be called \emph{$k$-rectangles} appearing in $x$ (see figure below; a $3$-rectangle in some array $x$ is shown in grey).
\begin{center}
\includegraphics[width=12cm]{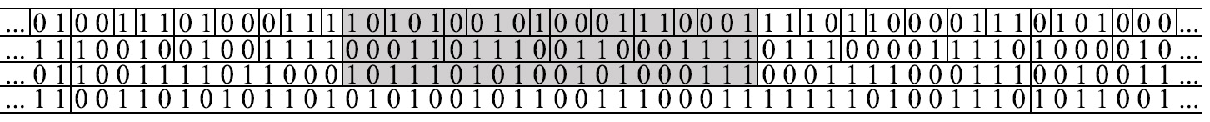}
\end{center}
\begin{proof}[Proof of Theorem \ref{markers}] We begin with some array representation of $X$ and we select a fast growing \sq\ $n_k$. For each $k$ we find an $n_k$-marker set and denote it by $F_k$. In every orbit we mark the visits to $F_k$ by placing markers in form of vertical bars in row $k$: if $T^nx\in F_k$ we place the bar next to (on the right of) the symbol $x_{k,n}$. In this manner, every array has markers in each row $k$ at distances bounded below by $n_k$ and above by $N_k=2n_k-1$ (so that (2) is satisfied). This representation is a (conjugate to \xt) array system over the enriched alphabets $\Lambda'_k = \Lambda_k\times\{\emptyset,|\} = \{a,a|: a\in\Lambda_k\}$. Next, we apply one of the above mentioned ``marker manipulations'' called \emph{upward adjustment}. We do not move the markers in row 1. Proceeding inductively on $k$, we move each marker in row $k+1$ horizontally so it matches (appears at the same horizontal coordinate as) the nearest to the left marker in row $k$ (if $\s=\na_0$ and there is no marker on the left in row $k$, we delete the marker in row $k+1$). It is easy to see that the procedure is shift-equivariant (if we moved the markers to the right, shift-equivariance would fail in case $\s=\na_0$) and for each marker, its new position depends on a bounded area in the array. So the resulting new representation is conjugate. The markers now satisfy (1), (which means that the corresponding new marker sets $F_k$ now form a nested \sq). If $n_k$ grows fast enough (we need $n_{k+1}$ to be much larger than $N_k$), then the gaps between the new markers in row $k$ range between new constants $n_k^{\min}$ and $n_k^{\max}$ which differ from the original bounds $n_k$ and $N_k$ only by a small percentage. So the condition still (2) holds.

\smallskip
In order to fulfill (3) we need another manipulation on the markers, which we will call \emph{subdividing}, and which should be applied prior to the upward adjustment. We will first do it only for $\z$-actions.

For each $n_k$ there exists a number $m_k$ (more precisely, this is $n_k(n_k+1)$) such that every number
$m\ge m_k$ can be represented as a sum $pn_k+q(n_k+1)$, where $p,q$ are nonnegative integers. Moreover, if we maximize $p$ (and minimize $q$), the pair $p(m),q(m)$ will be uniquely determined by $m$. We begin by placing in each row $k$ the markers corresponding to the visits in some $m_k$-marker sets (we call them \emph{primary markers}). This produces a conjugate representation of \xt. The primary markers divide the $k$th row of every array into intervals of bounded lengths larger than or equal to $m_k$. Next, we subdivide each of these intervals (by adding new \emph{secondary markers} between the primary markers) as follows: if $m$ is the length of the interval, we subdivide it into $p(m)$ intervals of length $n_k$ on the left, followed by $q(m)$ intervals of length $n_k+1$ on the right. The position of each secondary marker depends on a bounded area around it (and the primary markers) and is shift-equivariant for $\s=\z$, so the new representation remains conjugate. After this modification we disregard the classification into primary and secondary markers, and we consider all of them simply as \emph{markers} (or $k$-markers if the row number is specified). Now, the gaps between the $k$-markers range between $n_k^{\min}=n_k$ and $n_k^{\max}=n_k+1$, so (2) and (3) obviously hold.

At this point we apply the upward adjustment, which (assuming that $n_k$ grows fast enough) changes each of the numbers $n_k^{\min}$ and $n_k^{\max}$ by a small percentage, so that (2) and (3) still hold, while we also satisfy (1).

The above procedure fails for $\s=\na_0$ because the block in row $k$ between the coordinate $0$
and the first primary marker may be ``incomplete'', i.e., have no marker on the left end and be shorter than $m_k$. In such case we do not know how to subdivide this block (any decision in this aspect may lead to violation of shift-equivariance). So, we apply the following trick: for each $k$ we subdivide all but the initial (incomplete) block between the primary $k$-markers and then we shift all markers in that row by $M_k$ to the left, where $M_k$ is the maximal distance between the primary markers in row $k$. In this manner we get rid of the undivided interval via a continuous and shift-equivariant procedure. Now we apply the upward adjustment.
\end{proof}

We will also need the following observation: Some arrays have a ``marker of infinite depth'', i.e., extending throughout all rows. These however can be made exceptional:

\begin{lem}\label{meager}
Let \xt\ be an aperiodic zero-dimensional system. Then there exists a system of markers satisfying the conditions (1) and (2) of Theorem \ref{markers} and such that the set $M$ of arrays which have a marker of infinite depth is of I-st category.
\end{lem}
\begin{proof}
By the Baire theorem it suffices to arrange that the set $F_\infty = \bigcap_k F_{n_k}$ of arrays with a marker of infinite depth at the coordinate $0$ is of I-st category. Clearly, $F_\infty$ is closed, so we only need to arrange that it has empty interior. First observe that $F_\infty$ is always visited by each orbit at most once, because the gap between any two such visits must be larger than or equal to $n_k^{\min}$ for every $k$. Suppose $F_\infty$ has a nonempty interior $U$. We can represent $U$ as an increasing union of clopen sets, $U =\bigcup_k U_k$. Let $F'_{n_k}=F_{n_k}\setminus U_k$. Clearly, this is a nested \sq\ of clopen sets. Each orbit visits $F'_{n_k}$ at the same times as $F_{n_k}$ except at most one time, which implies that the gaps between the visits are bounded from below by ${n'}_k^{\min}=n_k^{\min}$ and above by ${n'}_k^{\max} = 2n_k^{\max}$. Thus the system of sets $F'_{n_k}$ induces a new system of markers satisfying the conditions (1) and (2) of Theorem~\ref{markers}. Moreover, we have $F'_\infty = \bigcap_k F'_{n_k} = F_\infty\setminus U$, which has empty interior.
\end{proof}

\begin{ques}
Can the system of markers in Lemma \ref{meager} be also made balanced?
\end{ques}

\section{Application of markers to entropy and vertical data compression}

By ``vertical data compression'' we shall mean a conjugate representation of a subshift (or another zero-dimensional system) using the smallest possible number of symbols. This procedure is in a sense orthogonal to the usual ``horizontal'' data compression used in information theory, where blocks are shrunk in length (without losing the information contents). Our vertical compression maintains the lengths, instead, if we imagine symbols of a large alphabet as high $\{0,1\}$-valued columns, it reduces the ``height'' of the representation. This justifies our terminology.
\smallskip

But first we need to devote few pages to entropy. Recall, that \tl\ entropy of a subshift \xt\ equals
$$
\htop(X,T)=\lim_n \frac1n\log(\#\B_n(X)),
$$
(we let ``$\log$'' denote the logarithm to base $2$), where $\B_n(X)$ is the collection of all blocks of length $n$ appearing in (some elements of) $X$.

\begin{thm}
Let \xt\ be an aperiodic zero-dimensional system equipped with a standard system of markers satisfying (1)-(3). Denote by $\R_k(X)$ the collection of all $k$-rectangles appearing in $X$. Then, if the numbers $n_k^{\min}$ grow sufficiently fast, we have
$$
\htop(X,T) = \lim_k \frac1{n_k^{\min}} \log(\#\R_k(X)).
$$
\end{thm}

\begin{proof}
First of all, it is true that the entropy of $(X,T)$ represented as an inverse limit $\overset{\longleftarrow}\lim_k(X_{[1,k]})$ equals $\lim_k\htop(X_{[1,k]})$ (this limit is nondecreasing).
In order to estimate $\htop(X_{[1,k]})$ we need to estimate the number of all rectangles of some large length $N$ appearing in the top $k$ rows of $X$. Any such rectangle is covered by at most $\frac N{n_{k'}^{\min}}+1$ concatenated $k'$-rectangles. Thus the wanted number of blocks of length $N$ does not exceed
$$
(\#\R_{k'}(X))^{\frac N{n_{k'}^{\min}}+1}.
$$
Taking logarithm, dividing by $N$, passing to a limit in $N$, then applying $\liminf$ in $k'$, and finally applying the limit over $k$, we get
$$
\htop(X)\le \liminf_{k'}\frac 1{n_{k'}^{\min}}\log(\#\R_k(X)).
$$
Now, by change of notation, we can skip the ``prime'' over $k$.
\smallskip

For the converse inequality, we know that given a decreasing to zero \sq\ $\epsilon_k>0$, for every $k$, if $n$ is large enough (larger than some $n(k)$), then
$$
\htop(X_{[1,k]}) \ge \frac1n\log(\#\B_n(X_{[1,k]}))-\epsilon_k.
$$
We must now assume that $n_k^{\min}\ge n(k)$ (this is the meaning of the assumption that $n_k^{\min}$ grow ``sufficiently fast''). Clearly, the collection of all blocks of length $n_k^{\max}$ in the top $k$ rows is at least as rich as $\R_k(X)$. This yields
$$
\htop(X_{[1,k]}) \ge \frac1{n_k^{\max}}\log(\#\R_k(X))-\epsilon_k.
$$
Letting $k$ pass to infinity we get
$$
\htop(X) \ge \limsup_k\frac1{n_k^{\max}}\log(\#\R_k(X)).
$$
By (3), we can replace $n_k^{\max}$ by $n_k^{\min}$. This ends the proof.
\end{proof}

\begin{exam} The assumption on the growth of $n_k^{\min}$ (even with (3) fulfilled) is necessary. Here is an appropriate example: The structure of markers is arbitrary for which the condition (1)-(3) are fulfilled. For instance, we can let it be as in the dyadic odometer: $n_k^{\max}=n_k^{\min}= 2^k$. The alphabets $\Lambda_k$ are also completely arbitrary (at least two-element), it will be convenient to assume that all of them contain $0$. To be specific, let $\Lambda_k=\{0,1,\dots,k\}$. We define $\Lambda_k'$ the usual way by adding the markers. Now we specify the ``allowed'' $k$-rectangles, by the following, very simple rule: only one $k'$-block (for only one $k'\le k$) inside the $k$-rectangle is not filled with zeros. However, this unique block (call it the \emph{free block}) can be arbitrary over $\Lambda_{k'}$ (naturally with the marker at the end). We define $X$ as the array system consisting of points whose all rectangles are allowed. The number of $k$-rectangles (whose lengths are $2^k$) exceeds $k^{2^k}$ (it suffices to count the rectangles with the free block in the last row), hence $\frac1{n_k^{\min}}\log(\#\R_k)=\log k$ tends to infinity. On the other hand, it is seen that in any element there are at most two free blocks (possibly in two different rows), and only if there exists a marker of infinite depth (otherwise, the free block is unique), so the system is strongly proximal
(all orbits converge both forward and backward to the unique fixpoint, the zero array). The entropy of such a system is zero.
\end{exam}

\begin{rem}\label{qq}
If \xt\ is a subshift then the assumption on the growth of $n_k^{\min}$ can be skipped. We leave the easy argument to the reader.
\end{rem}

\begin{rem}\label{qr}
If the system of markers does not satisfy (3), one can still calculate the \tl\ entropy by counting $k$-rectangles. Let $\R^n_k(X)$ denote the collection of $k$-rectangles occurring in $X$, whose length
is precisely $n\in[n^{\min}_k,n^{\max}_k]$. Then one has:
$$
\htop(X,T) = \lim_k\ \max\Bigl\{ \frac1n\log(\#\R^n_k(X)): n^{\min}_k\le n\le n^{\max}_k\Bigr\}.
$$
The proof is a bit complicated and we will skip it.
\end{rem}

\begin{thm}[Vertical Data Compression, see \cite{Kr82}]
Let \xt\ be an aperiodic subshift whose \tl\ entropy is $h$. Then \xt\ is conjugate to a subshift on
$\ell$ symbols, where $\ell$ is the smallest integer strictly larger than $2^h$.
\end{thm}

\begin{proof}
Let $\epsilon>0$ be such that $\ell>2^{h+\epsilon}$ and let $k$ be so large that
$$
\frac1{n_k^{\min}}\log(\#\R_k(X)) < h+\epsilon, \text{ \ \ \ i.e., \ \ \ } \#\R_k(X) < 2^{n_k^{\min}(h+\epsilon)}.
$$
We will only use $k$-markers for this particular index $k$, and we place them in the first (unique) row of the subshift. Now $k$-rectangles are the same as $k$-blocks. We will need the following combinatorial fact
(comp. \cite[Exercise 3.8]{Do11}): there exists a \emph{recognizable family}\footnote{A family of blocks is recognizable, if any two-sided or even one-sided (starting with a possible incomplete block) concatenation of its members decomposes in a unique way.} of blocks over $\ell$ symbols such that the cardinalities $c(n)$ of blocks
of length $n$ in this family satisfy
$$
\lim \frac1n\log c(n) \to \log \ell.
$$
This implies that eventually $c(n)$ is larger than $2^{n(h+\epsilon)}$. Thus, for large $k$, the numbers $c(n_k^{\min}), c(n_k^{\min}+1),\dots, c(n_k^{\max})$ are larger than $\#\R_k(X)$ and hence there exists an injection $\Phi$ assigning to each $k$-rectangle (i.e., a $k$-block) $R$ some block $\Phi(R)$ of length $|R|$ from the recognizable family. Now it suffices to replace, in every $x\in X$ the consecutive $k$-blocks by their images by $\Phi$. Because of the recognizability property, this coding is reversible, so we have produced a conjugate subshift on $\ell$ symbols, as required.
\end{proof}

\begin{ques}
How can the aperiodicity assumption be relaxed? Clearly, any subshift over $\ell$ symbols has at most $\ell^n$ periodic points with period $n$. Would this be a sufficient restriction?
\end{ques}

\begin{thm}[Vertical data compression of zero-dimensional systems]
Every aperiodic and zero-dimensional $\z$-action is conjugate to a ``subshift'' over the countable alphabet $\{1,2,\dots,\infty\}$ (the one-point compactification of $\na$).
\end{thm}

\begin{proof} This theorem appears as an exercise in \cite[Exercise 7.3]{Do11}. Here we provide a complete proof. Enumerate the set of all $k$-rectangles $\R=\R_1(X)\cup\R_2(X)\cup\cdots$ by natural numbers, bijectively and non-decreasingly in length. We compactify $\R$ by one point. This is going to be our countable alphabet $\Lambda$.

We define the map $\phi$ from $X$ into the shift over $\Lambda$ by describing the image $y=\phi(x)\in\Lambda^\z$ of every $x\in X$. We will encode $x$ ``row after row''. We encode the first row by placing in $y$, at the positions of all $1$-markers in $x$, the labels representing the $1$-blocks (equivalently $1$-rectangles, the elements of $\R_1(X)$) that follow these markers in $x$. Since we can assume that $n_1^{\min}\ge 2$, every sector in $y$ between the positions of two consecutive $1$-markers has at least one unfilled position. We now encode the second row of $x$ by placing in $y$, in the first empty slot between two $2$-markers of $x$, the labels representing the 2-blocks sitting there in the second row of $x$. Assuming that  $n_2^{\min}\ge 2n_1^{\max}$, after this step every sector in $y$ between two $2$-markers has at least one empty slot. We continue in this manner through all rows. All eventually unfilled positions in $y$ we fill with the infinities. It is clear that so defined map $x\mapsto y$ is continuous: every symbol (except the infinity) in $y$ is determined by a bounded rectangle in $x$ (i.e., its preimage is clopen). The infinity alone is not an open set, while any open neighborhood of the infinity is a complement of finitely many other symbols, so its preimage is also a clopen set. So the map is continuous. It is evident that so defined map $\phi$ commutes with the shift transformation (this would fail for $\s=\na_0$ because of the initial incomplete $k$-rectangle). To see that it is injective, note that we can easily reconstruct from $y$ the consecutive rows of $x$, as follows: For $k=1$, we locate in $y$ the symbols corresponding to the elements $\R_1(X)$. Their positions determine the $1$-markers and the symbols themselves provide information about the contents of the corresponding $1$-blocks in $x$. We continue inductively: Suppose the $k$th row of $x$ has been reconstructed (together with the $k$-markers). We locate in $y$ all symbols labeling the elements of $\R_{k+1}(X)$, and then we ``unload'' their contents each time starting at the nearest $k$-marker to the left, where we also place a $(k\!+\!1)$-marker. So, the map $\phi$ is a \tl\ conjugacy of $X$ with its image.
\end{proof}

To see that periodic points are an obstacle, take the identity map on the Cantor set. Every point is a fixpoint, so in any ``subshift'' (even with an infinite alphabet) it must be represented by a sequence filled with one symbol. Thus, uncountably many symbols are needed to encode all points.

\smallskip
To see how the above fails in non-injective systems, consider a system in which a Cantor set is sent by $T$ to one point (say $x$). No matter how we encode the system as a unilateral shift, the \sq\ representing $x$ must admit uncountably many shift-preimages, that is one-coordinate prolongations to the left. So, an uncountable alphabet is needed.

\begin{ques}
How can the aperiodicity assumption be relaxed? Of course, any subshift over a countable alphabet has at most countably many periodic points. Would this be a sufficient assumption?
\end{ques}

\section{Application of markers to the creation of minimal models}

Some ``block manipulations'' (or ``block codes'') depend on a bounded area in the array. These are continuous. For example, suppose that $\Lambda_k=\{0,1\}$ and that we want to flip (replace zeros by ones and vice versa) the symbols which fall at the positions of the markers. Not only this is a continuous, but also an invertible code. However, sometimes we need block manipulations which are not exactly continuous only ``finitary'', i.e., in ``most'' arrays they depend on a finite (but not uniformly bounded) area. This leads to ``exceptional arrays'', in which the manipulation cannot be determined by any finite area. In such arrays we usually have several choices of how they should be transformed. These are discontinuities of the algorithm. For example, suppose that we want to flip, in the first row, these symbols which fall at positions of markers whose depths are odd. In most points we will know exactly what to do, but there are points in which there occur markers of infinite depth. At such points the algorithm is not determined. Nevertheless, because such an infinite marker may appear in an array at most once, the set of discontinuities of such an algorithm has universal measure zero (we will say that they form a \emph{null set}). This is the meaning of a ``finitary algorithm'': the set of discontinuities is null. Most codes that refer to markers in distant rows are finitary (because markers in a distant row are far apart). At discontinuity points we give up defining the images so that finitary maps are left undefined on a null set. As we want the image to be a \ds, we define it as the closure of the image of the set of arrays on which the map is defined. This leads to the following definition:

\begin{defn}
Let \xt\ and \ys\ by \tl\ \ds s. By a \emph{finitary factor map} we mean any equivariant, partially defined and not necessarily surjective, map $\phi:X'\to Y$ defined and continuous on an invariant and dense\footnote{Invariance can be easily achieved by reducing $X'$ to a slightly smaller set. Density can be satisfied by replacing $X$ by $\overline{X'}$. Since the latter set supports all invariant probability measures, these changes are inessential from the point of view of invariant measures.} subset $X'\subset X$ with null complement.
By the \emph{image} or \emph{finitary factor of \xt} we mean the closure $\overline{Y'}$ in $Y$, where $Y'=\phi(X')$.
\end{defn}

\begin{fact}
The finitary factor is invariant under $S$ and $Y'$ is a set of full \im\ in $\overline{Y'}$.
\end{fact}

\begin{proof}
Invariance is obvious by continuity of $S$. Consider the graph $\Phi$ of $\phi$ and its closure $\overline\Phi$, together with the product action. For obvious reasons, $\overline\Phi$ is a \tl\ \ds\ and the projection $\pi_1$ on the first axis is a factor map onto \xt. Since $\phi$ is continuous at all points of $X'$, the set $\overline\Phi$ enlarges $\Phi$ only by containing more points projecting to $X\setminus X'$. This implies that $\Phi$ equals the $\pi_1$-preimage of $X'$. So, it is a set of full \im\ in $\overline\Phi$, while $B=\overline\Phi\setminus\Phi$ is a null set in $\overline\Phi$.
Now observe the projection $\pi_2$ on the second axis. Since $\pi_2(\Phi) = Y'$ and $\pi_2(\overline\Phi) = \overline{Y'}$, we have that $\pi_2^{-1}(\overline{Y'}\setminus Y')\subset B$. By completeness of the sigma-algebras, we get that $\overline{Y'}\setminus Y'$ is a null set.
\end{proof}

\begin{defn}
A \emph{finitary isomorphism} is an equivariant map which is defined, continuous, and injective on a dense invariant subset $X'\subset X$ of full \im.
\end{defn}

We remark that the inverse map to a finitary isomorphisms need not be a finitary isomorphism. It may happen that the inverse is so badly discontinuous that it remains discontinuous after removing any null set.
\smallskip

Finitary isomorphisms are almost as good as isomorphic extensions. They share many of their excellent features, listed in Fact \ref{4.4} below. However, by being not exactly continuous, they slightly perturb the topological structure of the modeled system.

\begin{fact}\label{4.4}
If $\phi:(Y,S)\to(X,T)$ is a finitary isomorphism then the adjacent map $\phi$ on \im s is an affine homeomorphisms between the sets of \im s of these systems and $\phi$ also serves as a measure-theoretic isomorphism between the systems \xmt\ and \yns, where $\mu$ is any \im\ on $X$ and $\nu = \phi\mu$.
\end{fact}

Once again, the easy proof is left to the reader.
\smallskip

We use this concept to show that every aperiodic zero-dimensional system is ``very close'' to being minimal:

\begin{thm}[Finitary Minimal Models Theorem, see also~\cite{Do06}]
Every aperiodic zero-dimen\-sional system is finitarily isomorphic to a minimal zero-dimensional system.
\end{thm}

\begin{proof}
We give the proof for invertible systems only. The noninvertible case is more complicated (there are at least two ways of handling it, both require much effort to be explained in detail).

\smallskip
We begin with an arbitrary array representation of the given aperiodic zero-dimen\-sional system \xt\ (with no markers yet). Next we modify it by inserting plenty of empty rows, as follows: below the row 1 we insert one empty row, below rows 2 and 3 (which are now 3 and 4) we insert four empty rows, below rows 4 and 5 (which are now 9 and 10) we insert 10 empty rows, and so on. We keep the enumeration of the nonempty rows, while we label the empty rows as, for example, $1'$, $1'',2'',3'',4''$, $1''',\dots,10'''$, and so on (we are not going to refer to this labeling).

\smallskip
In order to determine the numbers $n_k$ (which later determine the numbers $n_k^{\min}$ and $n_k^{\max}$) we will use a XIX-th century concept of a Frobenius number. Let $n^{(1)},n^{(2)},\dots,n^{(m)}$ be some natural numbers and let $S$ be the generated semigroup. $S$ is called a \emph{numerical semigroup} if $\mathsf{gcd}(n^{(1)},\dots,n^{(m)})=1$. It is known that such a semigroup consists of all except finitely many natural numbers. The largest missing integer is called the \emph{Frobenius number}. We extend this notion to the case where the $\mathsf{gcd}$ is larger (say, it equals $g$). Then by the \emph{Frobenius number of $S$} we will mean the Frobenius number of $S/g$ multiplied by $g$.
\smallskip

We proceed inductively, in each step defining the numbers $n_{2k}$ and $n_{2k+1}$, distributing the $2k$-markers and $(2k\!+\!1)$-markers, and introducing some block manipulations. In this construction, we only care that the markers satisfy conditions (1) and (2) of Theorem \ref{markers}.
\smallskip

Step 0 is trivial: we choose $n_1$ arbitrarily and we distribute the $1$-markers arising from Theorem
\ref{kriegerg} in row number $1$.
\smallskip

Suppose we have completed step $k-1$, which includes the distribution and upward adjustment of the $(2k\!-\!1)$-markers in row $2k-1$. Let $n_{2k-1}^{(1)},\dots,n_{2k-1}^{(m_{2k-1})}$ be the resulting lengths of the $(2k\!-\!1)$-rectangles appearing in $X$ after this step, and let $s_{2k-1}$ denote the Frobenius number of the generated semigroup. Let $\mathbf R_{2k-1}$ be a concatenation of all $2k\!-\!1$-rectangles appearing in $X$ (each used exactly once, the order is inessential) and let $r_{2k-1} = |\mathbf R_{2k-1}|$ (the horizontal length). We now choose $n_{2k}$ so large that after the distribution and upward adjustment of the $2k$-markers (in row $2k$), the minimal distance $n_{2k}^{\min}$ is larger than $r_{2k-1}+s_{2k-1}$. Next we choose $n_{2k+1}$ so that after the distribution and upward adjustment of the $(2k\!+\!1)$-markers, we have $\frac{n_{2k+1}^{\min}}{n_{2k}^{\max}}>2^k$.
\smallskip

Now is the time for the block manipulation. In every $(2k\!+\!1)$-rectangle we find the rightmost $2k$-rectangle and contained in it concatenation of $(2k\!-\!1)$-rectangles (which stretch throughout all rows up to $2k-1$, including the added, initially empty rows, which at this moment are no longer empty). We ``move'' the contents of this concatenation (together with all the markers which appear there) to the so far empty rows lying between rows $2k-1$ and $2k$ (there are exactly as many empty rows as wee need for that). Next we use the free space so obtained and put there $\mathbf R_{2k-1}$ concatenated with a suitable collection of $2k\!-\!1$-rectangles selected to match the missing length
of the free space. This is possible due to the fact that the missing length is larger than the Frobenius number $s_{2k}$. We must consistently use the same concatenation for every space of the same length.\footnote{This is where the method fails in the noninvertible case. The space on which we perform this manipulation may be incomplete (cut at the horizontal zero coordinate) and then we do not know its complete length, hence we cannot decide which concatenation to choose.}
\smallskip

This completes the construction. We need to verify that we have defined a finitary isomorphism and that the image is minimal.
\smallskip

Notice that the ultimate image of an array is determined at every coordinate whenever every coordinate falls in the rightmost $2k$-rectangle of a $(2k\!+\!1)$-rectangle for only finitely many $k$'s. For any invariant measure the probability of this happening at the vertical coordinate zero is at most
$\frac{n_{2k}^{\max}}{n_{2k+1}^{\min}}<2^{-k}$. By summability and Borel--Cantelli Lemma, the probability of this happening at the coordinate zero for infinitely many $k$'s is zero. Thus the probability of this happening for infinitely many $k$'s at some coordinate is also zero. At every other point (array) in $X$, at every coordinate the code is determined by a finite area (hence continuous) algorithm. At such points the code is also invertible, because the markers allow us to locate the artificially added blocks, and replace them by the original blocks stored in the additional rows. So, we have indeed defined a finitary isomorphism.

\smallskip
Minimality of $\overline{Y'}$ is fairly obvious: all $(2k\!-\!1)$-rectangles have been ultimately created in step $k-1$. They are never cut or altered ever after. In step $k$ all of them are placed syndetically together with the repetitions of $\mathbf R_{2k-1}$ in every $(2k\!+\!1)$-rectangle, and because in later steps we never cut the latter rectangles, this property persists throughout all steps of the construction of $\phi$, and obviously passes to the elements in the closure of the image. This proves minimality.
\end{proof}

\section{Markers in presence of periodic points}

If the system contains a periodic orbit of minimal period $n$ then any set is visited by this orbit either with gaps not exceeding $n$ or never. So, there is no chance to have a marker set of any order
$\ge n$. In this case we have several choices, how to violate the standard distribution of markers.
We can either
\begin{itemize}
	\item allow the arrays representing $n$-periodic points to have no markers in rows
	$k$ with $n_k\ge n$, and points close to such, to have arbitrarily large gaps in these rows, or
	\item allow the arrays representing $n$-periodic points (or close to such) to have markers with gaps
	equal to $n$ in rows $k$ for which $n_k\ge n$ (i.e., too short).
\end{itemize}
The first approach can still serve for entropy estimations and proved useful in many other constructions. The proof of the statement below is identical as that of Theorem \ref{krieger}. The details are left to the reader. The second approach will be exploited in the next section to the creation of Bratteli--Vershik models.
\smallskip

We denote by $\Per, \Per_n, \Per_{[1,n]}$ the sets of all periodic points, the ones with minimal period $n$
and with minimal period at most $n$, respectively.

\begin{thm}[Krieger's Marker Lemma, see \cite{Bo83}]\label{kriegerg}
If \xt\ is a zero-dimensional system then for every $\epsilon>0$ and every natural $n$ there exists a clopen $n$-marker set $F$ such that, for some natural $N$,
\begin{enumerate}
	\item $F$ is $n$-separated,
	\item $F\cup T^{-1}F\cup \dots \cup T^{-N}F\supset X\setminus (\Per_{[1,n+1]})^\epsilon$,
\end{enumerate}
where $A^\epsilon$ denotes the $\epsilon$-neighborhood of a set $A$.
\end{thm}
In other words, points in the neighborhood of periodic orbits with small periods may have markers appearing with unbounded or even infinite gaps (i.e., going forward or backward the markers may be missing).

\section{Bratteli--Vershik models of zero-dimensional systems}
Unlike previous sections, this one, in addition to being a survey of well-known notions and facts, leads also to a new result concerning specific representations of zero-dimensional systems.
\smallskip

We now recall briefly the notion of a Bratteli--Vershik representation of a zero-dimensional system. For more details see \cite{HPS92} or the surveys~\cite{durand:2010,bezuglyi_karpel,P10,S00}, where different aspects of Bratteli diagrams are studied.
\smallskip

A Bratteli diagram is a graph $B=(V,E)$ whose set of vertices $V$ is organized into countably many disjoint finite subsets $V_0$, $V_1$, \dots\ called {\it levels}. The zero level $V_0$ is a singleton $\{v_0\}$. The set of edges $E$ of the diagram is organized into countably many disjoint finite sets $E_1, E_2,\dots$. Every edge $e\in E_k$ connects a \emph{source} $s=s(e)\in V_k$ with some \emph{target} $t=t(e)\in V_{k-1}$. Each vertex is a target of at least one edge and every vertex of each level $k>0$ is also a source of at least one edge. Multiple edges connecting the same pair of vertices are admitted. By a \emph{path} we will understand a finite (or infinite) \sq\ of edges $p=(e_1,e_2,\dots,e_l)$ (or $p=(e_1,e_2,\dots)$) such that $t(e_{k+1})=s(e_k)$ for every $k=1,\dots,l-1$ (or $k=1,2,\dots$). Then the target of $e_1$ will be referred to as the target of the path and (only for finite paths) the source of $e_l$ will be referred to as the source of the path.


\begin{defn}
Given a Bratteli diagram $B$, we define the \emph{path space} $X_B$ as the set of all infinite paths with target $v_0$. We endow $X_B$ with the topology inherited from the product space $\prod_k E_k$, where each $E_k$ is considered discrete. Clearly, $X_B$ is compact, metric and zero-dimensional.
\end{defn}

For instance, the distance between two different paths $p = (e_1, e_2, \ldots)$ and $p' = (e'_1, e'_2, \ldots)$ in $X_B$ can be defined as $\frac{1}{2^N}$ where $N$ is the smallest integer such that $e_N \neq e'_N$. 

\emph{Telescoping} is a transformation of the Bratteli diagram $B$ into $B'=(V',E')$ consisting in choosing an infinite sub\sq\ $V_{k_i}$ of the levels, starting with $V_{k_0}=V_0$, and claiming them the levels $V'_i$ of the new diagram, and declaring the paths with targets in $V_{k_i}$ and sources in $V_{k_{i+1}}$
to be the edges of the new diagram between the levels $V'_i$ and $V'_{i+1}$.

The following fact is obvious and we omit the proof.
\begin{fact}
If $B'$ is obtained from $B$ by telescoping then $X_{B'}$ and $X_B$ are homeomorphic.
\end{fact}

\begin{defn}
We will say that the diagram $B$ is \emph{simple}, if there is a telescoping leading to a diagram in which for every pair of vertices $t\in V'_i, s\in V'_{i+1}$, there exists at least one edge $e'\in E'$ with $s(e')=s$ and $t(e')=t$.
\end{defn}

\begin{defn}
By an \emph{ordered Bratteli diagram} $(B,<)$ we shall mean a Bratteli diagram $B$ with a specific partial order. For each vertex $v\in V_k$, where $k>0$, all edges $e$ with $s(e)=v$ are ordered linearly
(i.e., enumerated as $\{e^{(1)}e^{(2)},\dots,e^{(n(v))}\}$). Edges with different sources are incomparable.
\end{defn}

The above order allows to introduce a partial order among finite and infinite paths. Two finite paths are comparable if they have a common source and the same length. For such paths we can apply the inverse lexicographical order: a path $p=(e_1,e_2,\dots,e_l)$ precedes $p'=(e'_1,e'_2,\dots,e'_l)$ if there exists
an index $1\le i\le l$ such that $e_j = e'_j$ for all $j>i$ (then $s(e_i)=s(e'_i)$) and $e_i<e'_i$ (we admit $i=l$; then the first condition is fulfilled trivially). Two infinite paths are comparable if they have targets in the same level and they are \emph{cofinal}, i.e., they agree from some place downward. In such case we apply to them the same rule as described above.

\begin{defn} A finite or infinite path is called \emph{maximal} (\emph{minimal}) if it has no successor (predecessor). By compactness, one can show that at least one maximal and one minimal path in $X_B$ always exists. A Bratteli diagram is \emph{properly ordered} if $X_B$ contains a unique maximal and unique minimal path.
\end{defn}

Denote by $X_{\min}$ the set of all minimal paths of an ordered Bratteli diagram and by $X_{\max}$ the set of all maximal paths.

\begin{defn} On the path space $X_B$ of an ordered Bratteli--Vershik diagram $(B,<)$, there is a natural, partially defined transformation $T_V$, called the \emph{Vershik map}. It is defined on the set of all but maximal paths and it sends every such path to its successor. The range of the map is the set of all but minimal paths. The Vershik map is a homeomorphism between its domain and range.
\end{defn}

The main area of applicability of Bratteli--Vershik representations are minimal Cantor systems, which is due to the following theorem proved in \cite{HPS92}.

\begin{thm}
Suppose $(B,<)$ is a properly ordered simple Bratteli diagram. Then the Vershik map $T_V$ prolongs to a self-homeomorphism $\bar T_V$ of $X_B$, by sending the unique maximal path to the unique minimal path. The resulting \ds\ $(X_B,\bar T_V)$ is zero-dimensional and minimal (hence $X_B$ is either finite or homeomorphic to the Cantor set). Conversely, every minimal homeomorphism of a zero-dimensional compact metric space is \tl ly conjugate to the system $(X_B,\bar T_V)$ for some properly ordered simple Bratteli diagram $(B,<)$.
\end{thm}

In this survey, we do not want to restrict to minimal systems and thus we will consider more general Bratteli--Vershik models. We will admit multiple maximal and minimal paths and we will not require the diagram to be simple.
Below we present a few examples illustrating zero-dimensional dynamical systems and their Bratteli--Vershik representations.

\begin{exam}[Vershik homeomorphism on a non-simple properly ordered Bratteli diagram]\label{1}
In the following example (see the picture below), the diagram is stationary, i.e. for every $n \in \mathbb{N}$ the edges between levels $n+1$ and $n$ are drawn and numbered in the same way as the edges between levels 2 and 1. There is one maximal path and one minimal path and these two paths are equal: it is the path passing through the vertices $\{v_i\}_{i = 1}^\infty$. The Vershik map prolongs to a self-homeomorphism of $X_B$ by sending the unique maximal path to the unique minimal path, i.e. the map has one fixed point. The set $X_B$ has infinitely countably many isolated points and one accumulation point (the fixpoint).


\begin{figure}[ht]
\unitlength = 0.4cm
\begin{center}
\begin{graph}(6,10)
\graphnodesize{0.4}
\roundnode{V0}(3,9)
\nodetext{V0}(0.5,0.5){$v_0$}
\roundnode{V11}(1,6)
\nodetext{V11}(-0.8,0){$v_1$}
\roundnode{V12}(5,6)
\nodetext{V12}(0.9,0){$w_1$}

\edge{V11}{V0}
\edgetext{V11}{V0}{0}
\edge{V12}{V0}
\edgetext{V12}{V0}{0}

\roundnode{V21}(1,3)
\nodetext{V21}(-0.8,0){$v_2$}
\roundnode{V22}(5,3)
\nodetext{V22}(0.9,0){$w_2$}

\edge{V21}{V11}
\freetext(0.5,4.5){0}
\bow{V22}{V11}{-0.09}
\bow{V22}{V11}{0.09}
\bowtext{V22}{V11}{-0.09}{0}
\bowtext{V22}{V11}{0.09}{2}
\edge{V22}{V12}
\freetext(5.5,4.5){1}

\roundnode{V31}(1,0)
\nodetext{V31}(-0.8,0){$v_3$}
\roundnode{V32}(5,0)
\nodetext{V32}(0.9,0){$w_3$}

\edge{V31}{V21}
\freetext(0.5,1.5){0}
\bow{V32}{V21}{-0.09}
\bow{V32}{V21}{0.09}
\bowtext{V32}{V21}{-0.09}{0}
\bowtext{V32}{V21}{0.09}{2}
\edge{V32}{V22}
\freetext(5.5,1.5){1}

\freetext(1,-1){$\vdots$}
\freetext(5,-1){$\vdots$}

\end{graph}
\end{center}
\end{figure}

\medskip

\medskip
The corresponding zero-dimensional system can be also described as the ``sunny side up'' subshift generated by the \sq\ \ $\dots\!0001000\!\dots$. It is easy to show explicitly the conjugacy between the systems. Indeed, the vertical path which passes through the vertices $\{v_i\}_{i = 1}^\infty$ corresponds to the accumulation point $\dots\!000000\!\dots$. The vertical path which passes through the vertices $\{w_i\}_{i = 1}^\infty$ corresponds to the point  $\dots\!000100\!\dots$, where the 1 is positioned at the coordinate zero (this point lies in the biggest distance from the accumulation point). The path which passes through the edge labeled $0$ between $w_{n+1}$ and $v_n$ corresponds to the point in the subshift which has $1$ at the position $- n$. And the path which passes through the edge labeled $2$ between $w_{n+1}$ and $v_n$ corresponds to the point which has $1$ at the position $n$.
\end{exam}

\begin{exam}[Vershik homeomorphism on an ordered non-simple Bratteli diagram with two paths which are both maximal and minimal]\label{2} If in Example~\ref{1} there were two edges between the vertices $v_1$ and $v_0$ then the diagram would remain properly ordered and would model a topologically transitive dynamical system with a limit cycle consisting of 2 points. The following example shows another ordered Bratteli diagram which models such a dynamical system. The diagram below is stationary after telescoping with respect to even levels. It has two maximal paths which are at the same time minimal (these are the vertical paths on the left and on the right side of the diagram). The Vershik map prolongs uniquely turning these paths to a periodic orbit of period 2. 
The corresponding dynamical system can be modeled as the orbit closure of the point \ $\dots000111\dots$, but under a map which is not the standard shift. The map is the composition of the flip (which exchanges zeros and ones) and the shift. Another description of the system would be a subshift generated by the sequence $\dots010110101\dots$ (a sequence which has one ``defect'' preventing it from being a periodic sequence of period 2). We leave the description of the conjugacy as an exercise to the reader. 

\begin{figure}[ht]
\begin{center}
\unitlength = 0.4cm

\begin{graph}(10,12)
\graphnodesize{0.4}
\roundnode{V0}(5,12)

\roundnode{V11}(1,9)
\roundnode{V12}(5,9)
\roundnode{V13}(9,9)

\edge{V11}{V0}
\edgetext{V11}{V0}{$0$}
\edge{V12}{V0}
\edgetext{V12}{V0}{$0$}
\edge{V13}{V0}
\edgetext{V13}{V0}{$0$}

\roundnode{V21}(1,6)
\roundnode{V22}(5,6)
\roundnode{V23}(9,6)

\edge{V21}{V11}
\edgetext{V11}{V21}{$0$}
\edge{V22}{V11}
\edge{V22}{V12}
\edge{V22}{V13}
\edge{V23}{V13}
\edgetext{V13}{V23}{$0$}
\edgetext{V22}{V11}{0}
\edgetext{V22}{V13}{2}
\bowtext{V22}{V12}{0.12}{1}

\roundnode{V31}(1,3)
\roundnode{V32}(5,3)
\roundnode{V33}(9,3)
\edge{V31}{V21}
\edgetext{V31}{V21}{$0$}
\edge{V32}{V21}
\edge{V32}{V22}
\edge{V32}{V23}
\edge{V33}{V23}
\edgetext{V33}{V23}{$0$}
\edgetext{V32}{V21}{2}
\edgetext{V32}{V23}{0}
\bowtext{V32}{V22}{0.12}{1}


\roundnode{V41}(1,0)
\roundnode{V42}(5,0)
\roundnode{V43}(9,0)

\edge{V41}{V31}
\edgetext{V41}{V31}{$0$}
\edge{V42}{V31}
\edge{V42}{V32}
\edge{V42}{V33}
\edgetext{V42}{V31}{0}
\edgetext{V42}{V33}{2}
\edge{V43}{V33}
\edgetext{V43}{V33}{$0$}
\bowtext{V42}{V32}{0.12}{1}

\freetext(1,-1){$\vdots$}
\freetext(5,-1){$\vdots$}
\freetext(9,-1){$\vdots$}

\end{graph}
\vspace{0.7cm}
\end{center}
\end{figure}
\end{exam}

In this survey, we aim to study Bratteli diagrams such that
the Vershik map \emph{determines} a homeomorphism of $X_B$, as in the following definition:

\begin{defn}\label{decisiveset}
We say that an ordered Bratteli diagram $(B,<)$ is \emph{decisive} if the Vershik map prolongs in a unique way to a homeomorphism $\bar T_V$ of $X_B$. A zero-dimensional \ds\ \xt\ will be called \emph{Bratteli--Vershikizable} if it is conjugate to $(X_B,\bar T_V)$ for a decisive ordered Bratteli diagram $B$.
\end{defn}

For example, every minimal Cantor system is Bratteli--Vershikizable.
One of the easiest examples of a decisive Bratteli diagram is a dyadic odometer, i.e. the diagram with a single vertex $v_n$ on each level $n$ and two edges joining $v_{n+1}$ and $v_n$ for all $n$. Then the Vershik homeomorphism can be described as the dyadic adding machine acting on the set $\{0,1\}^{\mathbb{N}}$. Notice that the diagrams in  Examples~\ref{1},~\ref{2} are also decisive.

The following lemma holds:

\begin{lem}\label{decisiveness}
An ordered Bratteli diagram is decisive if and only if two conditions hold:

\begin{enumerate}
	\item the Vershik map and its inverse are uniformly continuous on their domains, and
	\item the set of maximal paths and the set of minimal paths either both have empty interiors, or both
	their interiors consist of just one isolated point.
\end{enumerate}
\end{lem}

\begin{proof}
If the Vershik map $T_V$ and its inverse are uniformly continuous, then $T_V$ can be uniquely prolonged to a homeomorphism $T$ between the closure of its domain and the closure of its range, which in one case are both equal to the whole space, and, in the other case, the domain misses one isolated maximal path $p_0$ and the range misses one isolated minimal path $p_1$. Then we can prolong $T_V$ to the whole space $X_B$
in a unique way, 
by sending $p_0$ to $p_1$.

On the other hand, if the domain of $T_V$ misses a larger open set $U$, and the inverse map misses a set $U'$ homeomorphic to $U$, then $U$ and $U'$ contain a pair of homeomorphic clopen sets, at least two points each, and the map $T_V$ can be prolonged to a homeomorphism in more than one way. Finally, if $U$ and $U'$ are not homeomorphic, then $T_V$ cannot be prolonged to a homeomorphism of $X_B$ at all.
\end{proof}

To better illustrate the notion, we give two examples, where decisiveness fails for two different reasons.

\begin{exam}[A non-decisive ordered Bratteli diagram: the Vershik map is not uniformly continuous. Such a diagram is not a model of any \ds\footnote{Where \ds s are continuous maps on compact spaces.}] The diagram below is stationary. It has two minimal and two maximal paths. The minimal paths are vertical, one of them passes through the vertices $\{v_n\}_{n=1}^\infty$ and the other passes through the vertices $\{w_n\}_{n=1}^\infty$. One of the maximal paths (denote it by $x$) passes through the vertices $\{v_{2n-1}, w_{2n}\}_{n = 1}^\infty$ and the other through $\{w_{2n-1}, v_{2n}\}_{n = 1}^\infty$. The sets $X_{\min}$ and $X_{\max}$ are homeomorphic, but the Vershik map cannot be prolonged to the whole set $X_B$. Indeed, consider an infinite path $y(N)$ which passes through the vertices $\{v_{2n-1}, w_{2n}\}_{n = 1}^N$ for some $N$ and through the vertices $v_{2N + 1}, v_{2N}$. Such a path is non-maximal and it's image under the Vershik map passes through vertices $\{w_n\}_{n = 1}^{2N + 1}$. Consider also an infinite path $z(N)$ which passes through the vertices $\{v_{2n-1}, w_{2n}\}_{n = 1}^N$ and through the vertex $w_{2N + 1}$. Such a path is also non-maximal and it's image under the Vershik map passes through vertices $\{v_n\}_{n = 1}^{2N}$. As $N$ grows to infinity, both $y(N)$ and $z(N)$ tend to $x$. 
Still the images of $y(N)$ and $z(N)$ under the Vershik map are far apart regardless of $N$ (they differ in the first level). Hence, the Vershik map is not uniformly continuous.

\begin{figure}[ht]
\unitlength = 0.4cm
\begin{center}
\begin{graph}(6,10)
\graphnodesize{0.4}
\roundnode{V0}(3,9)
\nodetext{V0}(0.5,0.5){$v_0$}
\roundnode{V11}(1,6)
\nodetext{V11}(-0.8,0){$v_1$}
\roundnode{V12}(5,6)
\nodetext{V12}(0.9,0){$w_1$}

\edge{V11}{V0}
\edgetext{V11}{V0}{0}
\edge{V12}{V0}
\edgetext{V12}{V0}{0}

\roundnode{V21}(1,3)
\nodetext{V21}(-0.8,0){$v_2$}
\roundnode{V22}(5,3)
\nodetext{V22}(0.9,0){$w_2$}

\edge{V21}{V11}
\freetext(0.5,4.5){0}
\edge{V22}{V11}
\edge{V21}{V12}
\freetext(2,3.8){1}
\freetext(4,3.8){1}
\edge{V22}{V12}
\freetext(5.5,4.5){0}

\roundnode{V31}(1,0)
\nodetext{V31}(-0.8,0){$v_3$}
\roundnode{V32}(5,0)
\nodetext{V32}(0.9,0){$w_3$}

\edge{V31}{V21}
\freetext(0.5,1.5){0}
\edge{V32}{V21}
\edge{V31}{V22}
\freetext(2,0.8){1}
\freetext(4,0.8){1}
\edge{V32}{V22}
\freetext(5.5,1.5){0}

\freetext(1,-1){$\vdots$}
\freetext(5,-1){$\vdots$}

\end{graph}
\end{center}
\end{figure}
\vspace{0.4cm}

This system can be modeled as a skew product (in fact a group extension) over a dyadic odometer, as follows, let $X = \{0,1\}^\mathbb{N}$ and $T \colon X \rightarrow X$ be the dyadic adding machine. The phase space of our system is $X\times \{0,1\}$ where $\{0,1\}$ is viewed as the group with addition modulo 2. The cocycle extension sends a point $(x, s) \in X \times \{0,1\}$ to a point $(Tx, s + f(x))$, where $f(x) = 0$ if and only if the first zero in the dyadic expansion of $x$ appears at an even position. The cocycle $f$ is discontinuous at the point $x=111,\dots$, i.e., the skew product transformation cannot be defined continuously at two points in the product space.

We note that, given a Bratteli diagram, the set of orders such that the Vershik map can be uniquely prolonged to a homeomorphism of the whole path-space of the diagram, is studied in \cite{BKY14,BY}.
\end{exam}

\begin{exam}(A non-decisive ordered Bratteli diagram: the Vershik map is uniformly continuous, but cannot be prolonged to a homeomorphism). The following diagram is also stationary and it differs from the diagram in Example~\ref{1} only by the order. It has one maximal and two minimal paths: the maximal path passes through the vertices $\{v_i\}_{i=1}^\infty$, this path is also minimal. The other minimal path passes through the vertices $\{w_i\}_{i=1}^\infty$.
Since $X_{\max}$ and $X_{\min}$ 
have different cardinalities, the Vershik map cannot be prolonged to a homeomorphism of the whole space $X_B$. Nonetheless, the Vershik map can be prolonged to a continuous non-invertible mapping of $X_B$ by sending the maximal path (which is also minimal) to itself.

The corresponding dynamical system can be modeled as the operation $n \mapsto n+1$ acting on the one-point compactification $\mathbb{N}_0 \cup \{\infty\}$. The vertical path which passes through the vertices $\{v_i\}_{i = 1}^\infty$ corresponds to the accumulation point $\{\infty\}$. The vertical path which passes through the vertices $\{w_i\}_{i = 1}^\infty$ corresponds to zero in $\mathbb{N}_0$. The path which passes through the edge labeled $1$ between $w_{n+1}$ and $v_n$ corresponds to the number $2n - 1$ and the path which passes through the edge labeled $2$ corresponds to the number $2n$. The path which corresponds to zero has no preimage under the prolonged Vershik map.

\begin{figure}[ht]
\unitlength = 0.4cm
\begin{center}
\begin{graph}(6,10)
\graphnodesize{0.4}
\roundnode{V0}(3,9)
\nodetext{V0}(0.5,0.5){$v_0$}
\roundnode{V11}(1,6)
\nodetext{V11}(-0.8,0){$v_1$}
\roundnode{V12}(5,6)
\nodetext{V12}(0.9,0){$w_1$}

\edge{V11}{V0}
\edgetext{V11}{V0}{0}
\edge{V12}{V0}
\edgetext{V12}{V0}{0}

\roundnode{V21}(1,3)
\nodetext{V21}(-0.8,0){$v_2$}
\roundnode{V22}(5,3)
\nodetext{V22}(0.9,0){$w_2$}

\edge{V21}{V11}
\freetext(0.5,4.5){0}
\bow{V22}{V11}{-0.09}
\bow{V22}{V11}{0.09}
\bowtext{V22}{V11}{-0.09}{1}
\bowtext{V22}{V11}{0.09}{2}
\edge{V22}{V12}
\freetext(5.5,4.5){0}

\roundnode{V31}(1,0)
\nodetext{V31}(-0.8,0){$v_3$}
\roundnode{V32}(5,0)
\nodetext{V32}(0.9,0){$w_3$}

\edge{V31}{V21}
\freetext(0.5,1.5){0}
\bow{V32}{V21}{-0.09}
\bow{V32}{V21}{0.09}
\bowtext{V32}{V21}{-0.09}{1}
\bowtext{V32}{V21}{0.09}{2}
\edge{V32}{V22}
\freetext(5.5,1.5){0}

\freetext(1,-1){$\vdots$}
\freetext(5,-1){$\vdots$}

\end{graph}
\end{center}
\end{figure}
\vspace{0.4cm}

Notice also that the set $X_{\min}$ has one isolated point and one non-isolated point of $X_B$ while the set $X_{\max}$ consists only of a non-isolated point.
\end{exam}

\medskip
Our goal is to give a sufficient condition for a system to be Bratteli--Vershikizable. In preparation, we provide a passage from an ordered Bratteli diagram to an array representation with markers. This method is taken from \cite{DM08}.

Let $(B,<)=(V,E,<)$ be an arbitrary ordered Bratteli diagram. With each vertex $v\in V_k$ we will associate a $k$-rectangle called a \emph{$k$-symbol}, which will be denoted by the same letter $v$.\footnote{The $k$-symbol is nothing else but a tower in the Kakutani-Rokhlin partition, except that we draw the towers horizontally, and our notation keeps track of how the tower traverses the towers of the preceding generation.} The $k$-symbol $v$ will have $k\!+\!1$ rows (because we enumerate the rows from zero) and width equal to the number of paths with source $v$ and target $v_0$. We declare that in every row number $k$ we will use the alphabet $V'_k = V_k\times\{\emptyset,|\}=\{v, v|: v\in V_k\}$. Here is how we proceed:

With the vertex $v_0$ we associate a $0$-symbol of width $1$ carrying the symbol $v_0|$. Suppose we have defined all $k$-symbols, each being a $k$-rectangle of width equal to the number of outgoing paths.
Then for each vertex $v\in V_{k+1}$ we create the $(k\!+\!1)$-symbol $v$ as follows: in the top $k$ rows it has a concatenation $u_1u_2\dots u_l$ of the $k$-symbols corresponding the targets of the edges with source $v$, arranged in the same order as the paths are ordered in the diagram. Notice that the width of this concatenation equals precisely the number of paths connecting $v$ with $v_0$. Finally, to this concatenation we append the row number $k\!+\!1$ filled with repetitions of the symbol $v$, except that the rightmost symbol is $v|$. The figure below shows a fragment of an ordered diagram and a $2$-symbol corresponding to the vertex $u_1\in V_2$.

\begin{figure}[ht]
\unitlength = 0.4cm
\begin{center}
\begin{graph}(19,15)
\graphnodesize{0.4}
\roundnode{V0}(10,14)
\nodetext{V0}(0.5,0.5){$v_0$}
\roundnode{V11}(1,7)
\nodetext{V11}(1.1,0){$w_1$}
\roundnode{V12}(7,7)
\nodetext{V12}(1.1,0){$w_2$}
\roundnode{V13}(13,7)
\nodetext{V13}(1.1,0){$w_3$}
\roundnode{V14}(19,7)
\nodetext{V14}(1.1,0){$w_4$}

\bow{V11}{V0}{0.06}
\edge{V11}{V0}
\bowtext{V11}{V0}{0.06}{0}
\edgetext{V11}{V0}{1}
\edge{V12}{V0}
\bow{V12}{V0}{0.09}
\bow{V12}{V0}{-0.09}
\bowtext{V12}{V0}{0.09}{0}
\bowtext{V12}{V0}{-0.09}{2}
\edgetext{V12}{V0}{1}
\bow{V13}{V0}{0.06}
\bow{V13}{V0}{-0.06}
\bowtext{V13}{V0}{0.06}{0}
\bowtext{V13}{V0}{-0.06}{1}
\edge{V14}{V0}
\bow{V14}{V0}{0.06}
\bow{V14}{V0}{-0.06}
\bow{V14}{V0}{-0.12}
\bow{V14}{V0}{-0.18}
\bowtext{V14}{V0}{0.06}{0}
\edgetext{V14}{V0}{1}
\bowtext{V14}{V0}{-0.06}{2}
\bowtext{V14}{V0}{-0.12}{3}
\bowtext{V14}{V0}{-0.18}{4}


\roundnode{V21}(6,0)
\nodetext{V21}(-0.8,-0.1){$u_1$}
\roundnode{V22}(14,0)
\nodetext{V22}(0.9,-0.1){$u_2$}

\edge{V21}{V11}
\bow{V21}{V11}{0.09}
\edge{V21}{V12}
\edge{V21}{V14}
\edge{V22}{V11}
\bow{V22}{V11}{0.09}
\bow{V22}{V12}{-0.06}
\edge{V22}{V13}
\edge{V22}{V14}
\freetext(4.5,1.2){1}
\freetext(5.1,1.4){3}
\freetext(6.2,1.3){0}
\freetext(7.1,0.6){2}

\freetext(12,0.4){0}
\freetext(12.6,0.8){1}
\freetext(13.2,1.2){2}
\freetext(13.9,1.3){3}
\freetext(14.8,1.2){4}

\freetext(-1,0){$V_2$}
\freetext(-1,7){$V_1$}
\freetext(-1,14){$V_0$}

\freetext(6,-1){$\vdots$}
\freetext(14,-1){$\vdots$}

\end{graph}
\end{center}
\end{figure}

\begin{align*}
&\boxed{v_0}\boxed{v_0}\boxed{v_0}\boxed{v_0}\boxed{v_0}\boxed{v_0}\boxed{v_0}\boxed{v_0}\boxed{v_0}
\boxed{v_0}\boxed{v_0}\boxed{v_0}\\[-1.5\jot]
&\boxed{w_2\ w_2\ w_2}\boxed{w_1\ w_1\!}
\boxed{\,w_4\ w_4\ w_4\ w_4\ w_4\,}\boxed{\!w_1\ w_1}\\[-1.5\jot]
&\boxed{\ u_1\ \,u_1\ \,u_1\ \ u_1\ \,u_1\ \ u_1\ \,u_1\ \,u_1\ \,u_1\ \,u_1\ \ u_1\ \,u_1\,}
\end{align*}

\bigskip
There is a small technical inconvenience concerning the markers enclosing the $k$-symbol on the left.
Formally, they do not belong to the $k$-symbol. But then, one $k$-symbol could match the suffix of another,
which we do not want to admit. This is why we will always picture the $k$-symbol with the left hand side markers included. In this way a $k$ symbol is never part of a wider one (in the wider one the last row
has no markers inside). When concatenating two $k$-symbols we will always ``glue'' the markers at the contact line and treat them as one marker.

With each path $p\in X_B$ we can now associate an array $x=x(p)=[x_{k,n}]_{k\ge 1,n\in I}$ which has infinitely many rows, while the column numbers range over some interval $I$ of integers, which can be either finite, extending from a nonpositive to a nonnegative number ($I=[n,m]$, $n\le 0\le m$), or one-way infinite ($I=(-\infty, m]$, $m\ge 0$ or $I=[n,\infty)$, $n\le 0$) or both-way infinite ($I=\z$). If $p=(e_1,e_2,\dots)$ then the array $x$ is obtained as the limit of appropriately aligned $k$-symbols $v_0,v_1,v_2,\dots$, where $v_k=s(e_k)$ for $k\ge 1$. The alignment is done according to the following rule: let $n_k$ denote the position of the finite path $(e_1,e_2,\dots,e_k)$ (the ``top'' of the path $p$) among all paths connecting the vertex $v_k$ with $v_0$. Then we place the $k$-symbol $v_k$ so that its $n_k$-th column (counting from the left) sits at the coordinate $0$ of the horizontal axis. It is obvious that so aligned $k$-symbols $v_k$ are consistent, i.e., $v_{k+1}$ completely covers $v_k$ and matches it on the overlap area. This implies that the limit array $x(p)$ is well defined. We endow the space $\tilde X_B=\{x(p): p\in X_B\}$ with the topology inherited from $X_B$ by the bijection $p\mapsto x(p)$.

It is also clear that a path $p$ is maximal (minimal) if and only if the range $I$ of the columns of $x(p)$ ends (starts) at the coordinate $0$. The Vershik map corresponds to the horizontal left shift of the arrays $x(p)$ (which is executable on an array if and only if the corresponding path is not maximal), and the inverse map---to the right shift (on arrays corresponding to all but minimal paths). So, we have a complete array representation of the Vershik map. Question is, when can one uniquely prolong it to all paths.

Now, if the system is Bratteli--Vershikizable, then to each maximal path we can associate a \emph{unique} minimal path which is its image in the prolonged Vershik map (and to each minimal path---a unique maximal path which is its preimage). This means that in the above array representation, to each array whose horizontal domain ends (starts) at zero, we can concatenate a unique array which starts at $1$ (ends at $-1$). By shifting, to each array whose horizontal domain ends (starts) at any place $n$, we can concatenate a unique array which starts at $n+1$ (ends at $n-1$). In this manner, each path $p$ is now represented by a unique full array (with horizontal domain $\z$) which we denote by $x(p)$. We skip the fairly obvious verification, that the two systems: $(X_B, \bar T_V)$ and $(X,T)$, where $X=\{x(p):p\in X_B\}$ and $T$ is the usual horizontal shift, are conjugate. Notice that we have obtained an array representation with a system of markers which is upward adjusted and such that the maximal gap is bounded in each row. The condition that $n^{\min}_k$ grows to infinity with $k$ need not be fulfilled.

We conclude with a new theorem in whose proof we will reverse the construction: from an array representation we will produce a decisive Bratteli--Vershik system.

\begin{thm}\label{deci}
Let \xt\ be an aperiodic zero-dimensional system. Then it admits a decisive Bratteli--Vershik representation.
\end{thm}

\begin{proof}
We begin with an array representation with markers satisfying (1) and (2) of Theorem \ref{markers} and the assertion of Lemma \ref{meager}. It seems that to define an ordered Bratteli diagram it suffices to use the $k$-rectangles as vertices and let the edges connect each $(k\!+\!1)$-rectangle $R$ to its component $k$-rectangles appearing in the top $k$ row, and finally to order the edges following the natural order of the $k$-rectangles in the concatenation appearing in $R$. Such a naive idea indeed produces an ordered Bratteli diagram $(B',<')$ such that the Vershik map (wherever defined) agrees with the shift $T$ on $X$. Unfortunately, the method does not produce a \emph{decisive} diagram. For that, we need to include, in each vertex of the diagram, some more ``information'' from the array representation.

Namely, we create somewhat artificial objects which we will call \emph{$k$-trapezoids}, as follows: a $k$-trapezoid (appearing in some $x\in X$) consists of a $k$-rectangle enlarged in rows $1$ through $k\!-\!1$ by two $(k\!-\!1)$-rectangles (one on each side), then, in rows $1$ through $k\!-\!2$ by two
more $(k\!-\!2)$-rectangles (one on each side), etc. The figure below shows a $3$-trapezoid.
\begin{center}
\includegraphics[width=12cm]{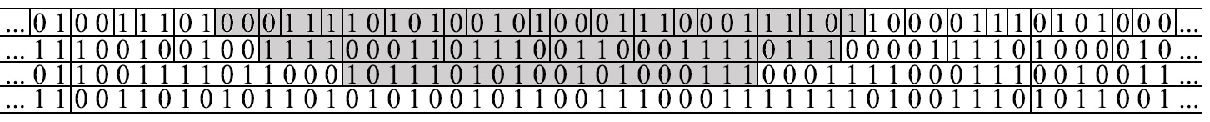}
\end{center}
Notice that while each $(k\!+\!1)$-rectangle $R$ has in its top $k$ rows a concatenation of $k$-rectangles, say $R_1,R_2,\dots,R_q$, the $(k\!+\!1)$-trapezoid $S$ which extends $R$ has in its top $k$ rows an ``overlapping concatenation'' of two more $k$-trapezoids, say $S_0,S_1,\dots,S_q,S_{q+1}$.
We will call the $k$-trapezoids $S_1,\dots,S_q$ \emph{internal}, while $S_0$ and $S_{q+1}$ will be called \emph{external}. The internal $k$-trapezoids extend the $k$-rectangles included in $R$, the external ones do not.

Now, we create a new ordered Bratteli diagram $(B,<)$: its vertices of level $k$ are the $k$-trapezoids, the edges connect a vertex corresponding to a $(k\!+\!1)$-trapezoid $S$ with its \emph{internal} $k$-trapezoids, the order of the edges is natural, as before. As a matter of fact, the new diagram projects to the naive diagram $(B',<')$ by identifying all vertices (i.e., $k$-trapezoids) which extend the same $k$-rectangle.

Now we will argue that the new diagram is decisive. Notice that this time each infinite path (including maximal and minimal) represents a \sq\ of $k$-trapezoids which converge to a \emph{full} array (with full horizontal domain $\z$; no arrays are produced whose horizontal domain is bounded on one or both sides). So, the natural mapping from the path space $X_B$ to the arrays is onto the array representation of $X$.
It is clear that this is a homeomorphism, the Vershik map (wherever defined) corresponds to the horizontal shift. Thus the Vershik map and its inverse are uniformly continuous. Moreover, maximal and minimal paths correspond to arrays with a marker of infinite depth (extending through all rows). Thus, by Lemma \ref{meager}, the set of such paths is of I-st category, hence its complement is dense and the two
conditions of decisiveness are fulfilled as in Lemma~\ref{decisiveness}.
\end{proof}

In~\cite{M06} it was proved that every aperiodic homeomorphism of a Cantor set has a Bratteli--Vershik representation such that the number of paths between any vertex $v \in V_n$ and the vertex $v_0$ tends to infinity as $n$ grows. In other words, every path in such diagram has infinitely many cofinal paths and the cofinal equivalence relation is aperiodic. The following example shows that such a diagram need not be decisive.

\begin{exam}\label{medynets}(A non-decisive ordered Bratteli diagram: the cofinal equivalence relation is aperiodic, yet the Vershik map can be prolonged in many different ways). 
Consider the following diagram.

\begin{figure}[ht]
\unitlength = 0.4cm
\begin{center}
\begin{graph}(30,15)
\graphnodesize{0.4}
\roundnode{V0}(15,14)
\freetext(15.5,14.6){$v_0$}
\roundnode{V11}(5,11)
\nodetext{V11}(-0.6,0.4){$u$}
\roundnode{V12}(15,11)
\nodetext{V12}(-0.6,0.4){$v$}
\roundnode{V13}(25,11)
\nodetext{V13}(0.6,0.4){$w$}

\edge{V11}{V0}
\edgetext{V11}{V0}{0}
\edge{V12}{V0}
\edgetext{V12}{V0}{0}
\edge{V13}{V0}
\edgetext{V13}{V0}{0}

\roundnode{V21}(3,6)
\roundnode{V22}(7,6)
\roundnode{V23}(15,6)
\roundnode{V24}(23,6)
\roundnode{V25}(27,6)


\bow{V23}{V12}{-0.06}
\bow{V23}{V12}{0.06}
\freetext(14.3,8.5){0}
\freetext(15.7,8.5){1}

\edge{V21}{V11}
\edgetext{V21}{V11}{0}
\edge{V22}{V11}
\edgetext{V22}{V11}{0}
\edge{V21}{V12}
\edgetext{V21}{V12}{1}
\edge{V22}{V12}
\edgetext{V22}{V12}{1}

\edge{V24}{V12}
\edgetext{V24}{V12}{0}
\edge{V24}{V13}
\edgetext{V24}{V13}{1}
\edge{V25}{V12}
\edgetext{V25}{V12}{0}
\edge{V25}{V13}
\edgetext{V25}{V13}{1}


\roundnode{V31}(2,1)
\roundnode{V32}(4,1)
\roundnode{V33}(6,1)
\roundnode{V34}(8,1)

\roundnode{V35}(15,1)

\roundnode{V36}(22,1)
\roundnode{V37}(24,1)
\roundnode{V38}(26,1)
\roundnode{V39}(28,1)

\freetext(2,0){$\vdots$}
\freetext(4,0){$\vdots$}
\freetext(6,0){$\vdots$}
\freetext(8,0){$\vdots$}
\freetext(15,0){$\vdots$}
\freetext(22,0){$\vdots$}
\freetext(24,0){$\vdots$}
\freetext(26,0){$\vdots$}
\freetext(28,0){$\vdots$}


\bow{V35}{V23}{-0.06}
\bow{V35}{V23}{0.06}
\freetext(14.3,3.5){0}
\freetext(15.7,3.5){1}

\edge{V31}{V21}
\edgetext{V31}{V21}{0}
\edge{V31}{V23}
\edgetext{V31}{V23}{1}

\edge{V32}{V21}
\edgetext{V32}{V21}{0}
\edge{V32}{V23}
\edgetext{V32}{V23}{1}

\edge{V33}{V22}
\edgetext{V33}{V22}{0}
\edge{V33}{V23}
\edgetext{V33}{V23}{1}

\edge{V34}{V22}
\edgetext{V34}{V22}{0}
\edge{V34}{V23}
\edgetext{V34}{V23}{1}

\edge{V36}{V24}
\edgetext{V36}{V24}{1}
\edge{V36}{V23}
\edgetext{V36}{V23}{0}

\edge{V37}{V24}
\edgetext{V37}{V24}{1}
\edge{V37}{V23}
\edgetext{V37}{V23}{0}

\edge{V38}{V25}
\edgetext{V38}{V25}{1}
\edge{V38}{V23}
\edgetext{V38}{V23}{0}

\edge{V39}{V25}
\edgetext{V39}{V25}{1}
\edge{V39}{V23}
\edgetext{V39}{V23}{0}

\end{graph}
\end{center}
\end{figure}

Any path in $X_B$ has infinitely many cofinal paths. Hence the Vershik map has infinite orbits and is aperiodic. Any path which passes through the vertex $u$ is minimal, while every path which passes through the vertex $w$ is maximal. There is also one more minimal and maximal paths passing through the vertex $v$. Hence the sets of minimal and maximal paths are homeomorphic and have non-empty interiors. In order to be continuous, the prolongation of the Vershik map should map the maximal path passing through $v$ to the minimal path passing through $v$. On remaining maximal paths it can be prolonged in many ways. This dynamical system can be also described as a map acting on the product of the Cantor set $X$ with the set $\{-1,-\frac12,-\frac13,\dots,0,\dots,\frac13,\frac12,1\}$. 
The map acts as follows: $(x, -\frac{1}{n}) \mapsto (Tx, -\frac{1}{n+1})$, $(x,0) \mapsto (Tx, 0)$, and $(x, \frac{1}{n}) \mapsto (Tx, \frac{1}{n-1})$, where $T$ is the dyadic adding machine acting on $X$. This map can be prolonged using any homeomorphism between the sets $(X, 1)$ and $(X,-1)$.
\end{exam}

Of course, in the context of Theorem \ref{deci} it is natural to ask about a condition equivalent to Bratteli--Vershikizability. The research in this direction is in progress and we expect to be able to give an answer soon.
\medskip

\textbf{Acknowledgement.}
The research of both authors is supported by the NCN (National Science Center, Poland) Grant 2013/08/A/ST1/00275


\end{document}